\documentclass[12pt]{article}

\usepackage[utf8]{inputenc}
\usepackage{amsfonts,amssymb,amsmath,amsthm,
  mathtools}
\usepackage[english]{babel}

\usepackage{authblk}
\usepackage[mathscr]{euscript}
\usepackage{fullpage}

\numberwithin{equation}{section}

\theoremstyle{plain}
\newtheorem{theorem}{Theorem}[section]
\newtheorem{lemma}[theorem]{Lemma}
\newtheorem{proposition}[theorem]{Proposition}
\newtheorem{corollary}[theorem]{Corollary}
\newtheorem{conjecture}[theorem]{Conjecture}

\theoremstyle{definition}
\newtheorem{definition}[theorem]{Definition}
\newtheorem{example}[theorem]{Example}
\newtheorem{remark}[theorem]{Remark}

\newcommand{\Kd}{\mathscr K^d}
\newcommand{\Ktwo}{\mathscr K^2}

\newcommand{\Rd}{\mathbb R^d}

\newcommand{\R}{\mathbb R}

\newcommand{\sC}{\mathscr C}
\newcommand{\sD}{\mathcal D}
\newcommand{\sB}{\mathcal B}
\newcommand{\sY}{\mathcal Y}
\newcommand{\sI}{\mathcal I}
\newcommand{\sJ}{\mathcal J}

\newcommand{\sM}{\mathcal M}
\newcommand{\sZ}{\mathcal Z}
\newcommand{\sX}{\mathcal X}

\renewcommand{\P}{\mathbf P}
\newcommand{\salg}{\mathfrak{F}}

\newcommand{\1}{\mathbf 1}
\newcommand{\eps}{\varepsilon}

\newcommand{\Hr}{\mathop{\overset{\smash{\lower0.3ex\hbox{$\,\,
          \scriptstyle\circ$}}}{H\!}}{}}
\renewcommand{\phi}{\varphi}

\renewcommand{\emptyset}{\varnothing}
\newcommand{\dH}{\mathsf{d}_{\mathrm{H}}}

\newcommand{\Prob}[1]{\mathbf{P}\left\{#1\right\}}
\newcommand{\E}{\mathbf{E}}

\DeclareMathOperator{\conv}{conv}
\DeclareMathOperator{\essinf}{ess\,inf}

\begin{document}

\title{Random valuations}

\author[1,2]{Andrii Ilienko}
\author[1]{Ilya Molchanov}
\author[1]{Tommaso Vison\`a}

\affil[1]{University of Bern}
\affil[2]{Igor Sikorsky Kyiv Polytechnic Institute}

\maketitle

\begin{abstract}
  A valuation is a finitely additive function on the family of compact
  convex sets in $\mathbb{R}^d$. We study non-negative infinitely
  divisible random valuations, with particular emphasis on monotone,
  $\sigma$-continuous models with independent increments along nested
  families. After separating the deterministic part, we show that the
  L\'evy measure of such a valuation is generated by pairs $(F,r)$,
  where $F$ is a non-empty closed convex set and $r>0$, with each pair
  contributing $r\mathbf{1}_{F\cap K\neq\emptyset}$. This yields a Poisson
  representation and an equivalent formulation through a pure-jump
  completely random measure on the space of closed convex sets. For
  stationary valuations, we derive a cylinder-Grassmannian
  representation of the L\'evy measure. In the stationary isotropic
  case, we obtain a McMullen-type decomposition, at the level of
  one-dimensional distributions, into independent components stable
  under dilation of the argument.
\end{abstract}

\\[6pt]

\section{Introduction}
\label{sec:introduction}

Let $\Kd$ be the family of convex bodies (i.e., compact convex sets)
in $\R^d$. While the empty set is typically not considered a convex
body, we adopt the convention that it is included in $\Kd$. A
\emph{valuation} $\phi$ is an additive map from $\Kd$ to an abelian
semigroup. Additivity means that
\begin{equation}
  \label{eq:additive}
  \phi(K \cup L) + \phi(K \cap L) = \phi(K) + \phi(L)
\end{equation}
for all compact convex sets $K$ and $L$ such that $K\cup L$ is also
convex, see \cite{MR3820854} and \cite[Chapter~6]{S14}. Additionally,
we will always assume that $\phi(\emptyset)=0$. While, in general, the
values of valuations may belong to any semigroup, most of the
literature on valuations focuses on valuations with values in the set
of real or complex numbers or in the family of compact convex sets
equipped with Minkowski addition. We consider only real-valued
valuations $\phi:\Kd\to\R$.

While valuations share with measures the finite additivity property,
their domain of definition is not an algebra. The additivity property
restricted to the case when $K\cup L$ is convex makes the study of
valuations complicated and causes a number of problems in their
characterisation. For instance, unlike measures, a non-negative
valuation is not necessarily monotone and there is no analogue of the
Jordan decomposition for signed valuations. Some valuations, such as
perimeter or surface area, are given by the Hausdorff measures of the
boundaries of sets. However, many valuations are far from being
measures, most importantly valuations given by $\phi(K)=V_d(K+A)$,
where $V_d$ is the volume and $K+A=\{x+y: x\in K, y\in A\}$ is the
Minkowski sum of $K$ and a fixed compact convex set $A$.

A common assumption in the theory of valuations is their invariance
under a group of transformations of $\R^d$. In most cases, valuations
are assumed to be translation invariant, meaning that
$\phi(K + x) = \phi(K)$ for all $x\in\R^d$, see
\cite{MR3820854}. While the family of translation-invariant valuations
is still not completely described, adding the rotation invariance
assump\-tion results in valuations invariant under all rigid motions,
which are well understood. By Hadwiger's theorem
\cite[Theorem~6.4.14]{S14}, any real-valued valuation on $\Kd$ that is
continuous in the Hausdorff metric and invariant under rigid motions
can be expressed as a weighted sum of the intrinsic volumes $V_i(K)$,
$i=0,\dots,d$. We refer to \cite{S14} for the definition of intrinsic
volumes and recall that $V_d$ is the Lebesgue measure, $V_{d-1}$ is
one half of the surface area, and $V_0(K)$ is the indicator of
$K\neq\emptyset$.

Each measure naturally defines a valuation by restricting its domain
to $\Kd$.  Random measures are the most important example of random
set functions, and the theory of random measures is rich and useful in
various applications, see \cite{kalle17}.

Random measures are additive processes indexed by (not necessarily
convex) sets.  An important family of random measures arises in the
theory of (weighted) empirical measures constructed as sums of weights
attached to locations in the space, see \cite{MR4628026} and recently
in \cite{MR4816007}. Particular problems which arise when the weights
are signed have been addressed in \cite{bass:pyke84}. Gaussian
measures on the family of convex sets have been considered in
\cite{MR346884} and \cite{MR537695}.

Dropping the additivity requirement of a set-indexed process results in
non-additive measures and capacities, which are also widely used in
mathematics, e.g., in large deviations, cooperative games, and
mathematical economics. Random non-additive measures 
appear in relation to the theory of extreme values, see
\cite{nor86b,ver97}.  Set-indexed martingales have been studied in
\cite{MR1733295}, without assuming any additivity.

In this paper we study \emph{random valuations}. In view of what has
been mentioned, one can regard them as sandwiched between random
measures and general set-indexed processes. 
From a pure mathematical perspective, random valuations can be
considered as valuations taking values in the additive group of random
variables. However, some probabilistic assump\-tions imposed later in
this paper become far more transparent if random valuations are
treated as stochastic processes on the space of convex bodies. 

Imposing invariance assumptions on random valuations \emph{pathwise}
makes it possible to utilise available representations of
deterministic valuations. In contrast, we study random valuations
without imposing any invariance on their paths.
Instead, we occasionally impose stationarity and isotropy in
distribution. These assumptions do not ensure the 
invariance of individual realisations and therefore do not permit the
direct application of deterministic characterisation results.

The family of valuations not satisfying any pathwise invariance
assumptions is very rich and is presently understood only in two
cases: valuations on the line and $\sigma$-continuous monotone
integer-valued valuations on the plane, see \cite{ilien:mol:vis25}.
Recall that $\phi$ is said to be $\sigma$-continuous if
$\phi(K_n) \to \phi(K)$ whenever $K_n \downarrow K$, see
\cite[Section~6.2]{S14}.

We recall a conjecture made in \cite{ilien:mol:vis25} concerning
deterministic valuations; it is supported by its version valid for
integer-valued $\sigma$-continuous monotone valuations on the plane.
A measure $\mu$ on $\sC$ is said to be \emph{locally finite} if $\mu$
has finite total variation on the families
\begin{equation}
  \label{eq:sCK}
  \sC_K=\{F\in\sC\colon F\cap K\ne\emptyset\}
\end{equation}
for all compact sets $K$. The valuation $\phi$ is said to be locally
finite if $|\phi(K)|$, $K\subset W$, $K\in\Kd$, is bounded for all
compact $W$.

\begin{conjecture}
  \label{conjecture:rep}
  Each locally finite $\sigma$-continuous valuation $\phi$ is
  representable as
  \begin{equation}
    \label{eq:3}
    \phi(K)=\mu(\sC_K),\quad K\in\Kd,
  \end{equation}
  for a locally finite signed measure $\mu$ on $\sC$.
\end{conjecture}

It should be noted that such a representation of a monotone $\phi$ is
not necessarily unique and may involve a signed measure on $\sC$, as
examples in \cite{ilien:mol:vis25} show. 

While a description of general random valuations is surely out of
reach, this paper studies infinitely divisible random valuations and,
in particular, those which have the property of independent
increments. It is important to stress that the independence of
increments differs from assuming that the values on disjoint sets are
independent --- we prove that the latter assumption singles out
valuations which are completely random measures. In
this connection, note that set-indexed L\'evy processes were studied in
\cite{bass:pyke84} and further in \cite{he:me13}, assuming that the
values on disjoint sets are independent.

We show that each monotone infinitely divisible random valuation
with independent increments and no Gaussian component admits a
representation in terms of a Poisson process on the product of the
family $\sC$ of non-empty closed convex sets and the positive real
half-line. 
In this way we settle a random variant of our conjecture,
providing a complete characterisation of random valuations which admit
representation \eqref{eq:3} with $\mu$ being a completely random measure on $\sC$. The
main reason why the random variant holds is the fact that the sum of
two independent random variables is not zero a.s., unless these two
random variables are degenerate.

Particular attention in this paper is devoted to \emph{stationary} random
valuations which are infinitely divisible, monotone, and have
independent increments.  While the additivity property is imposed
almost surely, stationarity is understood for distributions. It is
shown that such stationary random valuations have a representation in
terms of measures on the products of Grassmannians of different orders
and the real half-line.  McMullen’s decomposition theorem (see, e.g.,
Theorem~6.3.5 in \cite{S14}) states that every continuous
translation-invariant valuation admits a unique decomposition into the
sum of valuations homogeneous of degrees $i=0,\dots,d$.  In the random
setting, the homogeneity property is replaced by stability. We prove a
variant of McMullen's theorem for isotropic random valuations, showing
that each monotone $\sigma$-continuous
stationary isotropic infinitely divisible random valuation without Gaussian
component and vanishing deterministic part has the same one-dimensional distributions as the sum of
stable random valuations.

The paper is organised as follows. Section~\ref{sec:basic-prop-rand}
addresses general properties of additive stochastic processes on
$\Kd$, most importantly, their separability. Section~\ref{sec:supp-infin-divis} introduces the
infinite divisibility property of random valuations. The existence of
the corresponding L\'evy measure is derived by applying results from 
\cite{MR3857855} obtained for stochastic processes indexed by
arbitrary sets. The support of the L\'evy measure is crucial to
construct infinitely divisible valuations. It is shown in 
Section~\ref{sec:indep-incr} that assuming independence of increments
we can deduce that the L\'evy measure is supported by two-valued
valuations. 

While the family of two-valued valuations is still very large,
Section~\ref{sec:monot-id-valu} invokes the $\sigma$-continuity and
monotonicity assumptions to show that the L\'evy measure is supported
by valuations of the type $\psi(K)=r\1_{F\cap K\neq\emptyset}$ for
$r>0$ and closed convex sets $F$. This provides a general series
representation of $\sigma$-continuous monotone infinitely divisible
valuations with independent increments.

Section~\ref{sec:homogeneity} addresses stability properties of random
valuations, noting that the scaling can be applied either to the
values or to the argument of the valuation. Assuming 
stationarity, it is possible to further specify the representation of
random valuations as done in Section~\ref{sec:stat-rand-valu}.
Finally, Section~\ref{sec:stationary-stable-id} provides a random
variant of McMullen's theorem. 

\section{Basic properties of random valuations}
\label{sec:basic-prop-rand}

Consider the space $\Kd$ of compact convex sets in $\R^d$ (including
the empty set) equipped with the (extended) Hausdorff metric $\dH$
(assuming that $\dH(\emptyset,K)=\infty$ for each non-empty $K$).  Two
sets $K,L\in\Kd$ are said to be \emph{admissible} if their union is
convex. Let $\R^{\Kd}$ denote the space of all functions
$\phi:\Kd\to\R$, and let $\sB^{\Kd}$ be its cylindrical
$\sigma$-algebra. We fix a complete probability space
$(\Omega,\salg,\P)$.

\begin{definition}
  A \emph{random valuation} $\Phi$ is a stochastic process on $\Kd$
  such that $\Phi(\emptyset)=0$ a.s.\ and, for each fixed admissible
  $K,L\in\Kd$, we have
  \begin{displaymath}
    \Phi(K\cup L)+\Phi(K\cap L)=\Phi(K)+\Phi(L)\quad \text{a.s.}
  \end{displaymath}
  The distributions of $(\Phi(K_1),\dots,\Phi(K_n))$ for $n\ge1$,
  $K_1,\dots,K_n\in \Kd$, are called the finite-dimensional
  distributions of $\Phi$.
\end{definition}

A random valuation $\Phi$ is said to be \emph{separable} if there
exist a countable family $\sD\subset\Kd$ and a set
$\Omega_0\subset\Omega$ of full probability such that, for all
$K,L\in\Kd$ with $K\cup L\in\Kd$, there exist sequences $(K_n)$ and
$(L_n)$ from $\sD$ such that $K_n\cup L_n$ is convex,
$K_n\downarrow K$ and $L_n\downarrow L$, and $\Phi(K_n)\to \Phi(K)$,
$\Phi(L_n)\to\Phi(L)$, $\Phi(K_n\cup L_n)\to\Phi(K\cup L)$ and
  $\Phi(K_n\cap L_n)\to \Phi(K\cap L)$ as
$n\to\infty$ for all $\omega\in\Omega_0$.
The separability property implies
\emph{separability in probability}, obtained by replacing the
convergence for all $\omega\in\Omega_0$ with convergence in
probability.

\begin{lemma}
  \label{lemma:additive-functions}
  If a random valuation $\Phi$ is separable, then the realisations of
  $\Phi$ almost surely belong to the family of additive functions on
  $\Kd$.
\end{lemma}
\begin{proof}
  Note that $\Phi$ belongs to the family of
  additive functions if and only if
  \begin{displaymath}
    \sup_{K,L\in\Kd,K\cup L\in \Kd}
    |\Phi(K)+\Phi(L)-\Phi(K\cap L)-\Phi(K\cup L)|=0 \quad \text{a.s.}
  \end{displaymath}
  By the separability property, the supremum may be reduced to one
  over $K$ and $L$ from a countable family $\sD$.  Then use the fact
  that the expression under the supremum vanishes a.s.\ for all
  admissible $K$ and $L$ from $\sD$.
\end{proof}

Recall that there are two basic notions of continuity for
a deterministic valuation~$\phi$:
\begin{enumerate}
\item[(C1)] continuity in the Hausdorff metric;
\item[(C2)] $\sigma$-continuity, meaning that $\phi(K_n)\to\phi(K)$ whenever
  $K_n\downarrow K$. 
\end{enumerate}
It is clear that (C1) implies (C2). Furthermore, if $\phi$ is
monotone, then $\sigma$-continuity is equivalent to the upper
semicontinuity of~$\phi$.

\begin{lemma}
  If $\Phi$ is a separable random valuation, then
  $C_\Phi=\{\Phi\text{ is continuous}\}$ is an $\salg$-measurable event.
\end{lemma}
\begin{proof}
  Let $\tilde C_\Phi$ denote the event that $\Phi$ is locally
  uniformly continuous on $\sD$:
  \begin{displaymath}
    \tilde C_\Phi=\bigcap_{r\ge1}\bigcap_{m\ge1}\bigcup_{n\ge1}
    \bigcap_{\substack{K,L\in\sD,
        K,L\subset B_r,\dH(K,L)\le n^{-1}}}
    \bigl\{|\Phi(L)-\Phi(K)|\le m^{-1}\bigr\},
  \end{displaymath}
  where $B_r$ is the Euclidean ball of radius $r$ centred at the
  origin.  Since all intersections and unions are countable, we have
  $\tilde C_\Phi\in\salg$.

  To prove the lemma, it suffices to show that
  $\Omega_0\cap \tilde C_\Phi\subseteq \Omega_0\cap C_\Phi$. The reverse
  inclusion follows from the local compactness of $\Kd$ with respect
  to the Hausdorff metric: any continuous function on $\Kd$ is locally
  uniformly continuous there, hence, also on $\sD$.  To prove the
  direct inclusion, let $K_n\to K$. By the definition of separability,
  choose $K^i,K_n^i\in\sD$ with $K^i\downarrow K$ and $K_n^i\downarrow K_n$ such
  that $\Phi(K^i)\to\Phi(K)$ and $\Phi(K_n^i)\to\Phi(K_n)$ for all
  $\omega\in\Omega_0$. Hence,
  \begin{displaymath}
    |\Phi(K_n)-\Phi(K)|\le|\Phi(K_n)-\Phi(K_n^i)|
    +|\Phi(K_n^i)-\Phi(K^i)|+|\Phi(K^i)-\Phi(K)|,
  \end{displaymath}
  and, therefore,
  \begin{align*}
    \limsup_{n\to\infty}|\Phi(K_n)-\Phi(K)|
    =\limsup_{n\to\infty}\limsup_{i\to\infty}|\Phi(K_n)-\Phi(K)|
    \le \limsup_{n\to\infty}\limsup_{i\to\infty}|\Phi(K_n^i)-\Phi(K^i)|,
  \end{align*}
  which vanishes by the local uniform continuity of $\Phi$ on $\sD$.
  Since $\widetilde C_\Phi\in\salg$, the probability space is
  complete, and $\Omega\setminus\Omega_0$ is a null event, it follows
  that $C_\Phi\in\salg$.
\end{proof}

The following result implies that an a.s.\ $\sigma$-continuous random
valuation is separable.

\begin{lemma}
  \label{lemma:separable}
  Each random valuation whose realisations are almost surely
  $\sigma$-continuous is separable.
\end{lemma}
\begin{proof}
  Since the family $\sD$ of polytopes (including $\emptyset$) with
  rational vertices is countable and dense in $\Kd$, we only need to
  show that $\sD$ approximates pairs of convex sets with convex union
  by sequences which also have convex unions.
  Let $K,L\in\Kd$ be such that $U=K\cup L\in\Kd$. Choose a sequence
  $\varepsilon_n\downarrow0$ and decreasing sequences
  $U_n,K_n,L_n\in\sD$ such that
  \begin{displaymath}
    U\subset U_n\subset U+\eps_n B^d,\quad
    K+\eps_n B^d\subset K_n\subset K+2\varepsilon_n B^d,\quad
    L+\eps_n B^d\subset L_n\subset L+2\varepsilon_n B^d.
  \end{displaymath}
  Let $K'_n=K_n\cap U_n$, $L'_n=L_n\cap U_n$.  Then $K'_n,L'_n\in\sD$
  and $K'_n\downarrow K$, $L'_n\downarrow L$,
  $K'_n\cup L'_n\downarrow K\cup L$.
	
  It remains to check that $K'_n\cup L'_n\in\sD$. By the choice of
  $\eps_n$, we have $U_n\subset K_n\cup L_n$, hence
  \begin{displaymath}
    K'_n\cup L'_n=(K_n\cap U_n)\cup(L_n\cap U_n)=
    (K_n\cup L_n)\cap U_n=U_n\in\sD. \qedhere
  \end{displaymath}
\end{proof}

The family of all valuations is very rich, and some invariance
conditions are usually necessary to arrive at meaningful
characterisation results.
If almost all realisations of $\Phi$ are continuous and translation
invariant, then, by McMullen's theorem (see, e.g., Theorem~6.3.5 in
\cite{S14}), each realisation of $\Phi$ can be decomposed into the sum
of homogeneous valuations of orders $i=0,\dots,d$. The following
result shows that the homogeneous summands can be chosen as random
valuations and establishes their measurability.

\begin{proposition}
  \label{prop:Macmullen}
  Assume that $\Phi$ is a random valuation such that almost all of its
  realisations are continuous translation-invariant valuations. Then
  \begin{equation}
    \label{eq:macmullen}
    \Phi(K)=\sum_{i=0}^d \Phi_i(K),
  \end{equation}
  where each $\Phi_i$ is a random valuation such that almost all its
  realisations are continuous, translation-invariant valuations
  which are homogeneous of degree~$i$, that is,
  $\Phi_i(cK)=c^i\Phi_i(K)$ a.s.\ for all $K\in\Kd$ and $c>0$. 
\end{proposition}
\begin{proof}
  By McMullen's theorem, \eqref{eq:macmullen} holds for almost all
  realisations of $\Phi$. Each $\Phi_i$ is additive and homogeneous of
  degree $i$ on $\Kd$. For each $K\in\Kd$, we evaluate $\Phi$ at $rK$
  for $r>0$, which shows that
  \begin{displaymath}
    \Phi(rK)=\sum_{i=0}^d \Phi_i(r K)=\sum_{i=0}^d r^i \Phi_i(K).
  \end{displaymath}
  Thus, $\bigl(\Phi_0(K),\dots,\Phi_d(K)\bigr)$ is the solution of the
  above system of linear equations for $d+1$ distinct positive values
  $r_0,\dots,r_d$ of
  $r$. This solution is a linear transform of the values
  $\Phi(r_0 K),\dots,\Phi(r_d K)$, and so each $\Phi_i(K)$ is a random
  variable for all $K\in\Kd$.
\end{proof}

A direct consequence of the previous result and Hadwiger's theorem
(see \cite[Theorem~6.4.14]{S14}) is the following characterisation of
random valuations almost all of whose realisations are continuous and
invariant under rigid motions. In this case the random valuation is the
linear combination of intrinsic volumes $V_0,\dots,V_d$ with random
coefficients, see \cite[Section~4.1]{S14} for the definition of
intrinsic volumes.

\begin{corollary}\label{cor:hadwiger}
  Assume that $\Phi$ is a random valuation almost all of whose
  realisations are continuous and invariant under
  rigid motions. Then
  \begin{equation}
    \label{eq:hadwiger}
    \Phi(K)=\sum_{i=0}^d \xi_i V_i(K),
  \end{equation}
  where $(\xi_0,\dots,\xi_d)$ is a random vector, and $V_0,\ldots,V_d$
  are the intrinsic volumes.
\end{corollary}

A rich source of random valuations is provided by random (signed)
measures on the family $\sC$ of non-empty closed convex sets in $\R^d$ equipped
with the Borel $\sigma$-algebra generated by the Fell topology, see
\cite[Section~1.2]{ma75} and \cite[Appendix~C]{mo1}. 

\begin{proposition}
  \label{prop:Phi-Z}
  Let $Z$ be a locally finite random signed measure on $\sC$ (as
  defined in \eqref{eq:sCK}). Then
  \begin{equation}
    \label{eq:Z-representation}
    \Phi(K)=Z(\{F\in\sC\colon K\cap F\neq\emptyset\}),\quad K\in\Kd,
  \end{equation}
  is a $\sigma$-continuous random valuation. 
  Furthermore, $\Phi$ is a.s.~non-negative if $Z$ is
  a.s.~non-negative. In this case, $\Phi$ is necessarily monotone.
\end{proposition} 
\begin{proof}
  Since $Z$ is a random measure, $\Phi(K)$ is a random variable for
  every $K\in\Kd$ and thus $\Phi$ is a stochastic process on $\Kd$.
  Assume that $K,L\in\Kd$ are admissible.
  Since $Z$ is a random measure and its support is contained in the
  family of closed convex sets, 
  \begin{displaymath}
    \Phi(K)+\Phi(L)=Z(\sC_K)+Z(\sC_L)
    =Z(\sC_{K\cup L})+Z(\sC_{K\cap L})
    =\Phi(K\cup L)+\Phi(K\cap L)\quad\text{a.s.}
  \end{displaymath}
  Note that convexity of $K\cup L$ is essential above to conclude that
  $\sC_K\cap \sC_L=\sC_{K\cap L}$.  Since $\sC_\emptyset=\emptyset$,
  we have $\Phi(\emptyset)=Z(\emptyset)=0$ a.s.

  Assume that $(K_n)_{n\ge 1}$ is a sequence in $\Kd$ such that
  $K_n\downarrow K$. Since $Z$ is locally finite, its total variation on
  $\sC_{K_1}$ is finite. Since $\sC_{K_n}\downarrow \sC_K$ as
  $n\to\infty$, continuity from above for finite signed measures
  yields that $\Phi(K_n)=Z(\sC_{K_n})$ a.s.\ converges to
  $Z(\sC_K)=\Phi(K)$ as $n\to\infty$. 
  If $Z$ is non-negative, then $\Phi$ is also non-negative and
  monotone, since $\sC_L\subset \sC_K$ for $L\subset K$.
\end{proof}

\section{Infinitely divisible valuations}
\label{sec:supp-infin-divis}

A random valuation $\Phi$ is said to be \emph{infinitely divisible}
if, for each $n\geq 2$, $\Phi$ is equal in distribution to the sum
$\Phi_{1,n}+\cdots+\Phi_{n,n}$, where $\Phi_{1,n},\ldots,\Phi_{n,n}$
are $n$ i.i.d.~random valuations. We say that $\Phi$ is an \emph{ID
  valuation} if it is an infinitely divisible separable random
valuation. In the following we consider only non-negative $\Phi$,
which implies that its distribution does not have a Gaussian component. 

Let $\Phi$ be a non-negative ID valuation, and let
$\sI=\{K_1,\dots,K_m\}$ be a finite collection of sets from
$\Kd$. Then the random vector $\Phi_{\sI}=(\Phi(K_1),\dots,\Phi(K_m))$
is infinitely divisible. By the L\'evy--Khinchin representation (see
\cite{sato}), there exists a measure $\Lambda_\sI$ on
$\R_+^m\setminus\{0\}$ and $b_\sI\in\R_+^m$ such that, for each
$u\in\R_+^m$,
\begin{displaymath}
  \E \exp\Big\{-\big\langle u,\Phi_{\sI}\big\rangle\Big\}
  =\exp\bigg\{-\langle u,b_\sI\rangle
  -\int_{\R_+^m\setminus\{0\}} \Big(1-e^{-\langle u,x\rangle}\Big)
  \Lambda_{\sI}(dx)\bigg\}.
\end{displaymath}
The uniqueness of $\Lambda_\sI$ and $b_\sI$ implies
that for $\sI\subset\sI'$, the vector $b_{\sI'}$ restricted
to $\sI$ equals $b_\sI$, and $\Lambda_\sI$ is the projection of
$\Lambda_{\sI'}$ (on the space excluding the origin).

A measure $\Lambda$ on $(\R_+^{\Kd},\sB^{\Kd})$ is said to be a
\emph{L\'evy measure} if the following two conditions hold:
\begin{enumerate}
\item[(L1)] for all $K\in\Kd$,
  \begin{equation}
    \label{eq:Lambda-int-1}
    \int \min(1,\psi(K))\Lambda(d\psi)<\infty;
  \end{equation}
\item[(L2)] for all $A\in\sB^{\Kd}$,
  $\Lambda(A)=\Lambda_*(A\setminus \{O_{\Kd}\})$, where $\Lambda_*$ is the
  inner measure and $O_{\Kd}$ is the function $\psi\colon\Kd\to\R$ which is
  identically zero.
\end{enumerate}

Condition (L2) is satisfied if there exists a countable set
$\sJ\subset \Kd$ such that
\begin{displaymath}
  \Lambda\big(\{\psi\in\R_+^{\Kd}\colon\psi(K)=0\;\text{for all}\; K\in\sJ\}\big)=0.
\end{displaymath}

\begin{theorem}
  \label{thr:LH}
  Let $\Phi$ be a non-negative ID valuation.  Then there exists a
  unique (necessarily $\sigma$-finite) L\'evy measure $\Lambda$ on
  $\R_+^{\Kd}$ whose outer measure is concentrated on the family of valuations and a
  deterministic valuation $\phi$ such that, for any $m\ge1$, finite
  family $\sI=\{K_1,\dots,K_m\}\subset\Kd$, and $u\in\R_+^m$,
  \begin{displaymath}
    \E\exp\Big\{-\big\langle u,\Phi_{\sI}\big\rangle\Big\}
    =\exp\bigg\{-\langle u,\phi_{\sI}\rangle
    -\int_{\R_+^{\Kd}\setminus\{O_{\Kd}\}} \Big(1-e^{-\langle u,\psi_{\sI}\rangle}
    \Big)\Lambda(d\psi)\bigg\}.
  \end{displaymath}
  Furthermore, $\Phi$ is monotone if and only if $\phi$ is monotone
  and the $\Lambda$-outer measure is supported by monotone valuations.
\end{theorem}
\begin{proof}
  We let $\phi(K)=b_{\sI}$ with $\sI=\{K\}$ for $K\in\Kd$. Since
  $\Phi$ is additive on admissible sets, the uniqueness of $b_{\sI}$
  implies that $\phi$ is additive and so is a valuation. The
  additivity follows from the L\'evy--Khinchin formula applied to the
  vector $(\Phi(K),\Phi(L),\Phi(K\cup L),\Phi(K\cap L))$, which takes
  values from a linear subspace of $\R^4$. 

  By adapting the proof of Theorem~2.8 from \cite{MR3857855} to the
  family of all non-negative additive functions on $\Kd$, it is
  possible to patch together the measures $\Lambda_\sI$ to come up
  with a single measure on $\sB^{\Kd}$ which admits them as
  projections. The same works for the family of all monotone functions
  on $\Kd$. Indeed, all these properties of stochastic processes on
  $\Kd$ are determined by their finite-dimensional distributions.
  The $\sigma$-finiteness of $\Lambda$ follows from separability of
  $\Phi$ as in Corollary~2.18 from \cite{MR3857855} for
  general stochastic processes. 
\end{proof}

Let $N$ be a Poisson random measure on $(\R_+^{\Kd},\sB^{\Kd})$ with
intensity $\Lambda$, see \cite[Definition~15.2]{LastPenrose17}. Note
that $N\big(\{\psi\colon \psi(K)>\eps\}\big)$ has Poisson distribution
with parameter $\Lambda\big(\{\psi\colon \psi(K)>\eps\}\big)<\infty$,
and so is almost surely finite for each given $K\in\Kd$.  Since the
space of valuations with the pointwise convergence is not a standard
Borel space, it is not immediately possible to interpret the Poisson
random measure $N$ as a Poisson point process on the family of
valuations, see \cite[Section~2.1]{LastPenrose17}.  Proposition~2.10
from \cite{MR3857855} implies the following.

\begin{proposition}
  A non-negative ID valuation $\Phi$ with vanishing deterministic part
  has the same distribution as
  \begin{equation}
    \label{eq:Poisson-sum}
    \tilde{\Phi}(K)=\int \psi(K) N(d\psi),\quad K\in\Kd,
  \end{equation}
  where $N$ is the Poisson random measure on $(\R_+^{\Kd},\sB^{\Kd})$
  with intensity $\Lambda$; the outer measure associated with
  $\Lambda$ is concentrated on valuations.
\end{proposition}

The deterministic part $\phi(K)$ can be recovered as the essential
infimum $\essinf\Phi(K)$ for all $K\in\Kd$.
Unlike monotonicity and non-negativity, $\sigma$-continuity is not a
finite-dimensional constraint and therefore does not pass
immediately from $\Phi$ to its L\'evy measure.  The following result
addresses this issue.

\begin{lemma}
  \label{lemma:psi-sigma}
  If $\Phi$ is a monotone ID valuation whose realisations are almost
  surely $\sigma$-continuous, then its L\'evy measure is supported by
  $\sigma$-continuous valuations, meaning that the $\Lambda$-outer
  measure of the family of valuations that are not $\sigma$-continuous is zero.
\end{lemma}
\begin{proof}
  By Theorem~\ref{thr:LH}, the deterministic part $\phi$ of $\Phi$ is
  monotone and the $\Lambda$-outer measure is concentrated on monotone
  valuations.  Fix $K_0\in\Kd$ and $\eps>0$, and let $A_{K_0,\eps}$ be
  the family of all $\psi\in\R_+^{\Kd}$ such that $\psi(K_0)\geq\eps$.
  By \eqref{eq:Lambda-int-1}, $\Lambda(A_{K_0,\eps})<\infty$.  Use the
  Poisson representation of $\Phi$ and split its Poisson random
  measure into its restrictions to $A_{K_0,\eps}$ and its
  complement. This gives, on one probability space,
  $\widetilde\Phi=\phi+X+Y$, where
  \begin{displaymath}
    X(K) = \int_{A_{K_0,\eps}}\psi(K)N(d\psi)
  \end{displaymath}
  is a compound Poisson valuation and $Y$ is the remaining independent
  Poisson integral. The valuation $\widetilde\Phi$ has the same
  distribution as $\Phi$. Hence its outer distribution is concentrated
  on $\sigma$-continuous valuations.

  The three terms $\phi$, $X$, and $Y$ are monotone. Consider a
  realisation for which $\widetilde\Phi$ is $\sigma$-continuous, and
  let $K_n\downarrow K$. By monotonicity, the limits
  $\phi(K_n)\downarrow\phi_*$, $X(K_n)\downarrow X_*$, and
  $Y(K_n)\downarrow Y_*$ exist and satisfy $\phi_*\geq\phi(K)$,
  $X_*\geq X(K)$ and $Y_*\geq Y(K)$.  On the other hand, the
  $\sigma$-continuity of $\widetilde\Phi$ yields
  \begin{math}
    \phi_*+X_*+Y_*
    =
    \phi(K)+X(K)+Y(K).
  \end{math}
  Since all three defects are non-negative, each defect vanishes.
  Thus $X$ is $\sigma$-continuous.

  Let $\lambda=\Lambda(A_{K_0,\eps})$.  If $\lambda=0$, there is
  nothing to prove. If $\lambda>0$, condition on the event that the
  compound Poisson random measure defining $X$ has exactly one
  atom. On this event, $X=\psi$, where $\psi$ has distribution
  $\lambda^{-1}\Lambda|_{A_{K_0,\eps}}$. Since $X$ is almost surely
  $\sigma$-continuous, it follows that the outer measure associated
  with $\Lambda|_{A_{K_0,\eps}}$ is concentrated on
  $\sigma$-continuous valuations.

  Finally, let $\sD$ be the countable cofinal family of rational
  polytopes. Then the family of non-zero monotone valuations is a
  subset of the union of $A_{K_0,\eps}$ over all $K_0\in\sD$ and rational
  $\eps>0$. Thus, the $\Lambda$-outer measure of the family of
  valuations that are not $\sigma$-continuous is zero.
\end{proof}

\begin{example}
  Let $\{F_i,i\geq1\}$ be a point process on the family $\sC$ of
  non-empty closed convex sets in $\R^d$. If this process is locally finite,
  that is, only finitely many of the sets $F_i$ intersect any
  given compact set, then
  \begin{equation}
    \label{eq:sum-Fi}
    \Phi(K)=\sum_i \1_{K\cap F_i\neq\emptyset} 
  \end{equation}
  is an integer-valued monotone valuation. It is ID if the point
  process $\{F_i,i\geq1\}$ is infinitely divisible, meaning that, for
  any $n\geq2$, it can be represented as the superposition of $n$
  independent copies of another point process on $\sC$. It follows
  from \cite{ilien:mol:vis25} that on the plane all monotone
  $\sigma$-continuous integer-valued valuations are given by
  \eqref{eq:sum-Fi} with coefficients $r_i\in\{-1,1\}$ in front of
  the indicators.
\end{example}

\begin{example}
  If almost all realisations of $\Phi$ are continuous and invariant
  under rigid motions, then Corollary~\ref{cor:hadwiger} implies that
  it is infinitely divisible if and only if the vector of coefficients
  $(\xi_0,\dots,\xi_d)$ in \eqref{eq:hadwiger} is infinitely
  divisible.
\end{example}

\section{Independence of increments}
\label{sec:indep-incr}

As in the theory of random measures (see
\cite{kalle17}) and set-indexed processes (see \cite{MR1935676}), it
seems to be natural to consider valuations with completely independent
values, meaning that the values of $\Phi$ on disjoint sets are
independent. However, such random valuations are not very interesting
as Theorem~\ref{thr:indep-values} below shows.  In view of this, we impose
the following condition of independent increments.

A random valuation $\Phi$ is said to have \emph{independent
  increments} if, for each $n\geq 2$ and each nested sequence
$L_1\supset\cdots\supset L_n$ of compact convex sets and
$L_{n+1}=\emptyset$, the random variables $\Phi(L_i)-\Phi(L_{i+1})$,
$i=1,\dots,n$, are jointly independent.

If the paths of a non-trivial continuous random valuation $\Phi$ are
almost surely invariant under rigid motions, then
Corollary~\ref{cor:hadwiger} implies that $\Phi$ does not have
independent increments unless $\Phi(K)=\xi\1_{K\neq\emptyset}+\phi(K)$
for a random variable $\xi$ and a deterministic valuation $\phi$ which
is given by the weighted sum of the intrinsic volumes.  The same is
the case for random valuations whose realisations are continuous and
translation invariant, as the following result shows.

\begin{proposition}
  Assume that $\Phi$ is a continuous random valuation whose
  realisations are a.s.\ translation invariant. Then $\Phi$ has
  independent increments if and only if
  $\Phi(K)=\xi\1_{K\neq\emptyset}+\phi(K)$ where $\xi$ is a random
  variable and $\phi$ is a deterministic continuous
  translation-invariant valuation.
\end{proposition}
\begin{proof}
  By Proposition~\ref{prop:Macmullen}, $\Phi$ admits the
  representation \eqref{eq:macmullen}.  Since $\Phi_0$ in
  \eqref{eq:macmullen} is homogeneous of degree zero, we have
  $\Phi_0(K)=\Phi_0(cK)$ for all $K\in\Kd$ and $c>0$. Letting
  $c\downarrow 0$ and using continuity of $\Phi_0$, we obtain that
  $\Phi_0(K)=\Phi_0(\{0\})=\xi$ if $K\neq\emptyset$. Furthermore,
  $\Phi_i(\{0\})=0$ for $i=1,\dots,d$. 
  
  Fix $K\in\Kd$ and, by translation invariance, assume that $0\in
  K$. For $t>0$ amd $m\in\{1,\dots,d\}$, set
  \begin{gather*}
    X_t= \frac{1}{t^m} \Bigl( \Phi(tK)-\Phi(\{0\})
    -\sum_{i=1}^{m-1}t^i\Phi_i(K) \Bigr),\\
    Y_t= \frac{1}{t^m} \Bigl( \Phi(2tK)-\Phi(tK)
    -\sum_{i=1}^{m-1}(2^i-1)t^i\Phi_i(K) \Bigr).
  \end{gather*}
  Let $m=1$. 
  Since $\{0\}\subset tK\subset 2tK$, the two corresponding increments
  of $\Phi$ are independent if $m=1$. Hence, $X_t$ and $Y_t$ are
  independent. By homogeneity, $X_t\to\Phi_1(K)$ and
  $Y_t\to\Phi_1(K)$ a.s.\ as $t\downarrow0$, and so $\Phi_1(K)$
  is deterministic. This holds for all $K$ from a
  countable dense family, hence, for all $K\in\Kd$ by continuity. For
  $m\geq2$, we apply the induction argument. 

  Conversely, if $\Phi(K)=\xi\1_{K\neq\emptyset}+\phi(K)$ with
  deterministic $\phi$, then along any decreasing chain all increments
  of $\phi$ are deterministic, while at most one increment of the
  first term is random. Hence the increments are jointly independent.
\end{proof}

The following result shows that for infinitely divisible random
valuations it suffices to check the independence of two successive
increments associated with every nested family of four sets.

\begin{proposition}
  \label{prop:indep-increments}
  An infinitely divisible random valuation has independent increments
  if and only if $\Phi(K_2)-\Phi(K_1)$ and $\Phi(K_4)-\Phi(K_3)$ are
  independent for all compact
  convex sets $K_1\subset K_2\subset K_3\subset K_4$. 
\end{proposition}
\begin{proof}
  The result follows from the fact that for infinitely divisible
  random vectors joint independence is equivalent to pairwise
  independence, see Exercise~12.10(ii) from \cite{sato}. 
\end{proof}

\begin{lemma}
  \label{lemma:inf-div-indep}
  Let $\Phi$ be an ID valuation with the L\'evy measure $\Lambda$, and
  let $a_1,\dots,a_n$ and $b_1,\dots,b_m$ be real numbers and
  $K_1,\dots,K_n$ and $L_1,\dots,L_m$ be compact convex sets. Then the
  finite sums $\sum a_i\Phi(K_i)$ and $\sum b_j\Phi(L_j)$ are
  independent if and only if the $\Lambda$-outer measure vanishes on the
  complement of the set of valuations $\psi$ such that $\sum a_i\psi(K_i)=0$ or
  $\sum b_j\psi(L_j)=0$.
\end{lemma}
\begin{proof}
  The statement follows from the already mentioned Exercise~12.10 from
  \cite{sato}. 
\end{proof}

The following result shows that it is possible to weaken the
independence of increments property if $\Phi$ is assumed to be
monotone. 

\begin{proposition}
  \label{prop:weaker-indep}
  Let $\Phi$ be a monotone ID 
  valuation. Then $\Phi$ has independent increments if and only if
  $\Phi(K)-\Phi(L)$ and $\Phi(L)$ are independent for all $K,L\in\Kd$
  such that $L\subset K$.
\end{proposition}
\begin{proof}
  We need to prove only sufficiency. Let
  $K_1\subset K_2\subset K_3\subset K_4$.
  By assumption, $\Phi(K_4)-\Phi(K_2)$ and $\Phi(K_2)$ are independent.
  By Theorem~\ref{thr:LH} and
  Lemma~\ref{lemma:inf-div-indep}, $\Lambda$ is concentrated on
  monotone valuations $\psi$ such that $\psi(K_2)=0$ or
  $\psi(K_2)=\psi(K_4)$. 
  In the first case, monotonicity yields $\psi(K_2)-\psi(K_1)=0$,
  while in the second case
  $\psi(K_4)-\psi(K_3)=0$.
  Lemma~\ref{lemma:inf-div-indep} therefore implies that
  $\Phi(K_2)-\Phi(K_1)$ and $\Phi(K_4)-\Phi(K_3)$ are independent.
  The claim follows from Proposition~\ref{prop:indep-increments}.
\end{proof}

Next, we characterise ID valuations with independent values. Recall
that a random measure is said to be \emph{completely random} if its
values on disjoint sets are independent, see
\cite[Section~15.3]{LastPenrose17}. 

\begin{theorem}
  \label{thr:indep-values}
  Assume that $\Phi$ is a $\sigma$-continuous monotone ID random
  valuation with vanishing deterministic part and such that $\Phi(K)$
  and $\Phi(L)$ are independent for disjoint $K,L\in\Kd$. Then
  $\Phi$ has the same finite-dimensional distributions
  as the restriction to $\Kd$ of
  a completely random measure $\eta$ on $\R^d$.
\end{theorem}
\begin{proof}
  By Lemma~\ref{lemma:psi-sigma}, the L\'evy measure
  $\Lambda$ of $\Phi$ is supported by $\sigma$-continuous
  valuations. By Lemma~\ref{lemma:inf-div-indep}, outside of a set of
  zero $\Lambda$-outer measure, $\psi$ is monotone and
  $\sigma$-continuous and $\psi(K)=0$ or
  $\psi(L)=0$ for every disjoint pair from the countable family $\sD$
  of rational polytopes. This extends to arbitrary disjoint
  $K,L\in\Kd$ by taking their outer approximations. 
  
  For each $\psi$,   
  there exists $K\in\Kd$ such that
  $c=\psi(K)>0$. Without loss of generality, assume that $K$ is not a
  singleton. By splitting $K$ into three parts $L_1,L_2,L_3$ using two
  parallel hyperplanes, we see that $L_1$ and $L_3$ are disjoint. If
  $\psi(L_1)=0$, then the additivity and monotonicity imply
  $\psi(L_2\cup L_3)=\psi(K)$. If $\psi(L_1)>0$, then $\psi(L_3)=0$,
  consequently $\psi(L_1\cup L_2)=\psi(K)$. Continuing this procedure,
  we obtain a decreasing family of sets which shrinks to a point
  $x\in K$ with $\psi(\{x\})=\psi(K)=c$, also taking into account the
  $\sigma$-continuity of $\psi$. Then $\psi(L)=0$ for every $L\in\Kd$
  with $x\notin L$.

  Let $L\in\Kd$ contain $x$. By monotonicity, $\psi(L)\geq c$.
  Successively cut off convex pieces not containing $x$, retaining at
  each step a convex set $L_n\ni x$ with $\psi(L_n)=\psi(L)$. By
  $\sigma$-continuity, $\psi(L)=c$. Consequently,
  $\psi(L)=c\1_{x\in L}$ for all $L\in\Kd$. Thus, the $\Lambda$-outer
  measure is supported by valuations $\psi$ with this property. 

  Each $\psi$ of this form can be represented as a pair
  $(x,c)\in \R^d\times(0,\infty)$ and the map
  $(x,r)\mapsto r\1_{x\in K}$ and its inverse are measurable with
  respect to the cylindrical $\sigma$-algebra.  Thus, the L\'evy
  measure of $\Phi$ is the image of a measure on
  $\R^d\times(0,\infty)$ under this map.  Therefore, the Poisson
  random measure $N$ in \eqref{eq:Poisson-sum} can be realised as the
  pushforward of a Poisson point process $\{(x_i,r_i): i\ge1\}$
  on $\R^d\times(0,\infty)$, and thus $\Phi$ has the same distribution
  as  
  \begin{equation}
    \label{eq:points}
    \eta(K)=\sum_{i} r_i\1_{x_i\in K}, \quad K\in\Kd,
  \end{equation}
  which is a completely random measure
  restricted to $\Kd$, see \cite[Section~15.3]{LastPenrose17}.
\end{proof}

While the complete independence assumption reduces $\sigma$-continuous
monotone ID valuations to completely random measures, the independence
of increments property identifies a richer family of random
valuations, as the following section shows.

\section{Monotone $\sigma$-continuous ID valuations with independent
  increments}
\label{sec:monot-id-valu}

If $\Phi$ is monotone, then $\Phi(K)\geq \Phi(\emptyset)=0$ for all
$K\in\Kd$ and so $\Phi$ is non-negative.  For deterministic
valuations, the non-negativity property does not imply monotonicity,
e.g., for compact convex sets $M\subset L$, the valuation
$\phi(K)=\1_{K\cap L\neq\emptyset}-\1_{K\cap M\neq\emptyset}$ is
non-negative, but not monotone. However, such a situation is
impossible for ID valuations with independent increments.

\begin{proposition}
  If an ID valuation $\Phi$ with independent increments is
  non-negative and has vanishing deterministic part, then $\Phi$ is
  necessarily monotone.
\end{proposition}
\begin{proof}
  For $L\subset K$ put $X=\Phi(K)-\Phi(L)$ and $Y=\Phi(L)$. By
  independent increments, $X$ and $Y$ are independent. Since the
  deterministic part vanishes, the essential infimum of $Y$ is
  zero. Moreover, $Y\geq 0$ and $X+Y\geq 0$. If $X<-\eps$ with a
  positive probability, then $\Prob{Y<\eps}>0$ and independence would
  give that $\Prob{X+Y<0}>0$, a contradiction. Hence $X\geq 0$
  a.s. Separability then gives pathwise monotonicity on the common
  full-probability event. 
\end{proof}

The following result provides a complete description of
$\sigma$-continuous monotone ID valuations with independent
increments. Recall that $\sC$ denotes the family of all non-empty closed convex
sets in $\R^d$.

\begin{theorem}
  \label{thr:2}
  A monotone $\sigma$-continuous ID valuation $\Phi$ has independent
  increments if and only if its L\'evy measure $\Lambda$ is the
  pushforward of a Borel measure $\nu$ on $\sC\times(0,\infty)$ by the
  map
  \begin{equation}
    \label{eq:LM-indicators}
    (F,r)\mapsto \psi(K)=r \1_{K\cap F\neq\emptyset},\quad K\in\Kd.
  \end{equation}
  Furthermore, $\Phi$ has the same distribution as
  \begin{equation}
    \label{eq:positive-sum}
    \tilde{\Phi}(K)=\phi(K)
    +\sum_{i} r_i\1_{K\cap F_i\neq\emptyset},
    \quad K\in\Kd,
  \end{equation}
  where $\{(F_i,r_i),i\geq1\}$ is the Poisson point process on
  $\sC\times(0,\infty)$ with intensity $\nu$ and $\phi$ is a
  deterministic monotone $\sigma$-continuous valuation.  The
  measure $\nu$ necessarily satisfies
  \begin{equation}
    \label{eq:nu-levy}
    \int_{\sC\times(0,\infty)} \min(1,r)\1_{F\cap K\neq\emptyset}\,\nu(d(F,r))<\infty,
    \quad K\in\Kd. 
  \end{equation}
\end{theorem}
\begin{proof}
  Since the deterministic part $\phi(K)=\essinf\Phi(K)$ is a
  deterministic valuation, we replace $\Phi$ by $\Psi=\Phi-\phi$ and
  assume that $\Phi$ does not have a deterministic part.  If
  $K_n\downarrow K$, then 
  \begin{displaymath}
    [\phi(K_n)-\phi(K)]+[\Psi(K_n)-\Psi(K)]=\Phi(K_n)-\Phi(K)\to 0,
  \end{displaymath}
  so each of the non-negative terms on the left-hand side converges to
  zero, meaning that $\phi$ and $\Psi$ are $\sigma$-continuous.

  By Proposition~\ref{prop:weaker-indep} and
  Lemma~\ref{lemma:inf-div-indep}, $\Phi$ has independent increments
  if and only if $\Lambda$-outer measure is supported by $\psi$ such that
  \begin{equation}
    \label{eq:choice}
    \psi(K)=\psi(L) \quad \text{or}\quad \psi(L)=0
  \end{equation}
  for all $K,L\in\Kd$ with $L\subset K$. Let $K\in\Kd$ with
  $c=\psi(K)\neq 0$. If $K\subset W$ for $W\in\Kd$, then
  $\psi(W)=c$. For each $L\in\Kd$,
  take $W\in\Kd$ such that $(K\cup L)\subset W$. Then $\psi(L)$ is
  either zero or $\psi(L)=\psi(W)=c$. By applying this to a countable
  dense collection $K,L\in\sD$ and using $\sigma$-continuity of
  $\psi$, we obtain that the $\Lambda$-outer measure is supported by
  non-decreasing valuations $\psi$ with two values $\{0,c\}$ with
  $c>0$ depending on $\psi$. By Lemma~\ref{lemma:psi-sigma}, all such
  $\psi$ are $\sigma$-continuous.

  It follows from \cite{ilien:mol:vis25} that a non-negative
  $\sigma$-continuous valuation $\psi$ with two values is uniquely
  determined by its support, which is the set $F$ of all points $x$
  such that $\psi(\{x\})>0$. Hence,
  $\psi(K)=c\1_{F\cap K\neq\emptyset}$ for some $c>0$.  This set $F$
  is necessarily nonempty, since otherwise $\psi$ would have empty
  support and so would be zero. By $\sigma$-continuity of $\psi$ we
  immediately obtain that $F$ is closed. Indeed, if $x_n\in F$ and
  $x_n\to x$, we let $K_n$ be the convex hull of $\{x\}$ and
  $\{x_i,i\geq n\}$, so that $c=\psi(K_n)\to\psi(\{x\})$ and so
  $x\in F$.
  
  Let $L=\conv(\{x,y\})$ for any $x,y\in F$, and choose $z\in
  L$. Since $\psi$ is non-decreasing and takes only two values $0$ and
  $c$, we have $\psi(L)=c$. Similarly, $\psi(L_1)=\psi(L_2)=c$ for the
  segments with end-points $x,z$ and $y,z$, respectively. Additivity
  therefore yields that $\psi(\{z\})=c$, so $z\in F$.  Thus, $F$ is a
  convex set. The map between $\psi(K)=r\1_{F\cap K\neq\emptyset}$,
  $K\in\Kd$, with its cylindrical $\sigma$-algebra and the pair
  $(F,r)\in(\sC,(0,\infty))$ with the product $\sigma$-algebra
  is bimeasurable. Hence, the L\'evy measure of $\Phi$ corresponds to
  a measure on $\sC\times(0,\infty)$.
  
  Property \eqref{eq:nu-levy} follows from \eqref{eq:Lambda-int-1}.
  The space $\sC\times(0,\infty)$ is a standard Borel space and so a
  Poisson random measure with intensity $\nu$ can be realised
  as a Poisson point process $\eta$ on this space, see
  \cite{LastPenrose17}. Thus, 
  the integral in \eqref{eq:Poisson-sum} can be represented as the sum
  \eqref{eq:positive-sum}.

  In the reverse direction, the increments of $\tilde{\Phi}$ are
  determined by the Poisson process on $\sC\times(0,\infty)$ over
  disjoint sets and so are independent. 
\end{proof}

The Poisson process $\{(F_i,r_i),i\geq1\}$ on $\sC\times(0,\infty)$
from Theorem~\ref{thr:2} defines a pure-jump random measure
$Z$ on the family $\sC$ of closed convex sets in $\R^d$ by letting
\begin{displaymath}
  Z(\sM)=\sum_{i} r_i \1_{F_i\in\sM}
\end{displaymath}
for all measurable $\sM\subset\sC$. This random measure on $\sC$ is
completely random, meaning that its values on disjoint subsets are
independent, see \cite[Section~15.3]{LastPenrose17}. In contrast to
Theorem~\ref{thr:indep-values}, which relates monotone ID valuations with
independent values to a completely random measure on $\R^d$, the
random measure $Z$ is defined on the family $\sC$ of closed
convex sets in $\R^d$. Then, the representation in Theorem~\ref{thr:2}  can be
formulated as follows.

\begin{corollary}
  \label{cor:measure-Z}
  A non-negative $\sigma$-continuous random valuation $\Phi$ with
  $\essinf\Phi(K)=0$ for all $K\in\Kd$ is ID with independent
  increments if and only if $\Phi$ has the same distribution as
  $Z(\sC_K)$ for a non-negative pure-jump completely random measure
  $Z$ on $\sC$ with intensity $\nu$ satisfying \eqref{eq:nu-levy}.
\end{corollary}

Corollary~\ref{cor:measure-Z} settles a random variant of
Conjecture~\ref{conjecture:rep} for non-negative ID valuations with
independent increments. If $\Phi$ is integrable, it implies that
$\E\Phi(K)=\mu(\sC_K)$ for a measure $\mu$ on $\sC$. Given that, in
general, the measure $\mu$ representing the deterministic valuation
$\E\Phi(K)$ may be signed (see
Examples~3.3 and 3.4 in \cite{ilien:mol:vis25}), we conclude that not
every deterministic monotone $\sigma$-continuous valuation can be
obtained as the expectation of an ID valuation with independent
increments.

\begin{remark}[Special cases of $\nu$]
  If $\nu(\sC_K\times (0,\infty))<\infty$, then the sum in
  \eqref{eq:positive-sum} involves a finite number of terms for any
  given $K$. If $\nu(\sC\times(0,\infty))<\infty$, then the sum in
  \eqref{eq:positive-sum} involves only a finite number of terms
  simultaneously for all $K$.

  The ID valuation $\Phi$ with independent increments and vanishing
  deterministic part is a restriction to $\Kd$ of a random measure on
  $\R^d$ if and only if $\nu$ is supported by singletons, that is, by
  the family $\{(\{x\},r)\colon r\in(0,\infty), x\in\R^d\}$. In this
  case, the values of $\Phi$ on disjoint sets are independent.
\end{remark}

\begin{remark}[Uniqueness of $\nu$]
  In general, a measure $\nu$ on $\sC\times(0,\infty)$ is not uniquely
  determined by its values on $\sC_K\times A$ for all $K\in\Kd$ and
  Borel $A\subset(0,\infty)$, that is, by the valuation it
  generates. For example, on the line, let $\nu_1$ be the sum of Dirac
  measures at $([0,3],1)$ and $([1,2],1)$, and $\nu_2$ at $([0,2],1)$
  and $([1,3],1)$.  Then
  $\nu_1(\sC_K\times\cdot)=\nu_2(\sC_K\times\cdot)$ for all $K$.
  
  However, in case of ID valuations, the measure $\nu$ is determined
  uniquely by the distribution of the random valuation if one takes
  into account the finite-dimensional distributions of $\Phi$. First,
  for any $K\in\Kd$, the measure $\nu(\sC_K\times\cdot)$ is the L\'evy
  measure $M_K$ of the infinitely divisible random variable
  $\Phi(K)$. Second, by comparing the L\'evy measure $M_{K,L}$ of the
  random vector $(\Phi(K),\Phi(L))$ with the L\'evy measure of
  $(\sum_i r_i\1_{K\cap F_i\neq\emptyset},\sum_i r_i \1_{L\cap
    F_i\neq\emptyset})$, we see that the restriction of $M_{K,L}$ on
  the $x$-axis gives us $\nu((\sC_K\setminus\sC_L)\times\cdot)$, on
  the $y$-axis $\nu((\sC_L\setminus\sC_K)\times\cdot)$, and on the
  diagonal $\{(t,t): t>0\}$ (equivalently, to $(0,\infty)^2$)
  we obtain $\nu((\sC_K\cap\sC_L)\times\cdot)$. In this
  way, finite-dimensional distributions of $\Phi$ yield the values of
  the L\'evy measure $\nu$ on
  $(\sC_{K_1}\cap\cdots\cap\sC_{K_m})\times A$ for all
  $K_1,\dots,K_m\in\Kd$, $m\geq1$, and Borel
  $A\subset(0,\infty)$. Since this family of sets is closed under
  finite intersections, generates the Borel $\sigma$-algebra, and
  $\nu$ is finite on it outside the origin,
  the values of $\nu$ on this family determine
  uniquely $\nu$ on the whole Borel $\sigma$-algebra on the product
  space $\sC\times(0,\infty)$.
\end{remark}

\begin{example}
  \label{ex:M+x}
  Let $\nu$ be the pushforward of the product of the Lebesgue measure
  $V_d$ and a L\'evy measure $\theta$ on $(0,\infty)$ under the map
  $(x,r)\mapsto (x+M,r)$, where $M$ is a deterministic compact convex
  set. Furthermore, let $\phi(K)=0$ in \eqref{eq:positive-sum}. Then
  \begin{displaymath}
    \Phi(K)=\sum_i r_i\1_{K\cap (x_i+M)\neq\emptyset},
  \end{displaymath}
  where $\{(x_i,r_i),i\geq1\}$ is the Poisson point process on
  $\R^d\times(0,\infty)$ of intensity $V_d\otimes\theta$. Since 
  $K\cap M+x\neq\emptyset$ if and only if $x\in K+(-M)$,
  where $(-M)=\{-x:x\in M\}$, we have
  \begin{displaymath}
    \E e^{-u\Phi(K)}
    =\exp\Big\{-V_d(K+(-M))\int_0^\infty
    \big(1-e^{-ur}\big) \theta(dr)\Big\}.
  \end{displaymath}
  If $\theta=\delta_1$ is the Dirac measure at $1$, then $r_i=1$ for
  all $i$ and $\Phi(K)=\eta(K+(-M))$,
  where $\eta$ is a homogeneous Poisson random measure on
  $\R^d$.
  In analogy with the deterministic setting, where valuations
  $V_d(K+M)$, $M\in\Kd$, span a dense subset of the family of all
  continuous translation-invariant valuations (see
  \cite[Corollary~3.3]{MR3289841}), it is tempting to conjecture that
  linear combinations of random valuations $\eta(K+M)$, $M\in\Kd$, are
  dense (in the topology of pointwise a.s.\ convergence) in the family
  of all stationary $\sigma$-continuous monotone ID valuations with
  independent increments.
\end{example}

If the independence of increments is relaxed as described in
Appendix~A, then the L\'evy measure is supported by valuations
$\psi(K)$ given by suprema of upper semicontinuous quasi-concave
functions over $K$.
In the square-integrable case, the absence of correlation between
increments is equivalent to the fact that the variance of $\Phi$ is a
valuation itself, see Appendix~B.

\section{Homogeneity and stability properties}
\label{sec:homogeneity}

Below we consider only $\sigma$-continuous monotone valuations. 
The scaling of a valuation $\phi$ by a positive number $t$ can be
defined either by multiplying its values by $t$, which yields the
valuation $t\phi$, or by dilating the argument $K$ of $\phi$. The
latter scaling transforms the valuation $\phi$ into the valuation
$T_t\phi$ given by
\begin{displaymath}
  (T_t\phi)(K)=\phi(t^{-1}K), \quad K\in\Kd.
\end{displaymath}
A random valuation $\Phi$ is said to be \emph{$\beta$-homogeneous (in
  distribution)} with some $\beta\geq 0$ if $t^{-\beta}\Phi$ has the
same distribution as $T_t\Phi$ for all $t>0$.
Note that this concept of homogeneity is different from the
pathwise homogeneity, which appears in McMullen's representation from
Proposition~\ref{prop:Macmullen}.

A random valuation $\Phi$ is said to be \emph{$\alpha$-stable (in
  argument)} with some $\alpha\neq0$ if, for all $n\ge2$, the
valuation $T_{n^{-1/\alpha}}\Phi$ has the same distribution as the sum
of $n$ independent copies of $\Phi$. This immediately implies that
$\Phi$ is ID.  This $\alpha$-stability property is different from the
fact that the values $\Phi(K)$ are $\alpha$-stable random variables
for all $K\in\Kd$; the valuation $\Phi$ is said to be
\emph{$\alpha$-stable (in value)} if the sum of its $n$ independent
copies has the same distribution as $n^{1/\alpha}\Phi$.

\begin{proposition}
  \label{prop:stable}
  Let $\Phi$ be a non-trivial $\sigma$-continuous non-negative ID valuation with
  independent increments and vanishing deterministic part. Assume that
  the distribution of $\Phi$ is determined by the measure $\nu$ on
  $\sC\times(0,\infty)$, see Theorem~\ref{thr:2}.
  \begin{enumerate}
  \item[(i)] $\Phi$ is $\alpha$-stable in argument if and only if
    \begin{equation}
      \label{eq:nu-homogeneous}
      \nu(t^{1/\alpha}\sM\times A)=t \nu(\sM\times A)
    \end{equation}
    for all rational $t>0$, measurable $\sM\subset\sC$ and Borel
    $A\subset(0,\infty)$, and then $\alpha\in(0,\infty)$. Here the
    scaling of $\sM$ is applied to all its elements. 
  \item[(ii)] $\Phi$ is $\alpha$-stable in value if and only if 
    \begin{equation}
      \label{eq:nu-scaling-value}
      \nu(\sM\times tA)=t^{-\alpha}\nu(\sM\times A)
    \end{equation}
    for all $t>0$, measurable $\sM\subset\sC$ and Borel
    $A\subset(0,\infty)$,
    and then $\alpha\in(0,1)$.
  \item[(iii)] $\Phi$ is $\beta$-homogeneous if and only if 
    $\nu(t\sM\times A)=\nu(\sM\times (t^{-\beta} A))$ for all $t>0$,
    measurable $\sM\subset\sC$ and Borel $A\subset(0,\infty)$.
  \end{enumerate}
  In each assertion, it suffices to verify the corresponding identity for 
  $\sM=\sC_L$, where $L$ ranges over finite unions of compact convex sets
  and $A\subset[\eps,\infty)$ for any $\eps>0$. 
\end{proposition}
\begin{proof}
  All mentioned properties of $\Phi$ can be equivalently formulated
  for $\nu$.  If
  \begin{displaymath}
    \Phi(K)=\sum_i r_i\1_{K\cap F_i\neq\emptyset}, \quad K\in\Kd,
  \end{displaymath}
  then 
  \begin{displaymath}
    T_t\Phi(K)=\sum_i r_i\1_{K\cap tF_i\neq\emptyset}
    \quad\text{and}\quad
    t\Phi(K)=\sum_i tr_i\1_{K\cap F_i\neq\emptyset}.
  \end{displaymath}
  The intensity measure of the Poisson point process
  $\{(tF_i,r_i),i\geq1\}$ is given by
  $\nu((t^{-1}\sM)\times A)$ and of the Poisson point process
  $\{(F_i,tr_i),i\geq1\}$  by
  $\nu(\sM\times t^{-1}A)$. Furthermore, the L\'evy
  measure of the sum of $n$ i.i.d.\ copies of $\Phi$ is $n\nu$.

  \noindent
  (i) \textsl{Sufficiency} is obvious, since \eqref{eq:nu-homogeneous}
  says that the L\'evy measure of $T_{n^{-1/\alpha}}\Phi$ is exactly
  the L\'evy measure of $\Phi$ multiplied by $n$.  Hence,
  $\Phi(n^{1/\alpha}K)$ has the same distribution as the sum of $n$
  independent copies of $\Phi(K)$, and the same holds for
  finite-dimensional distributions. Applying this for $K$ which
  contains the origin and using the monotonicity of $\Phi$ show that
  $\alpha>0$.
  
  \noindent
  \textsl{Necessity.}  The stability property implies that the L\'evy
  measure $M_K$ of the random variable $\Phi(K)$ is homogeneous under
  scaling of $K$, namely, $M_{n^{1/\alpha}K}=nM_K$. Thus,
  $\nu(\sC_{n^{1/\alpha}K}\times A)=n\nu(\sC_K\times A)$ for all
  $K\in\Kd$ and Borel $A\subset(0,\infty)$.
  Note that
  \begin{displaymath}
    \sC_{n^{1/\alpha}K}
    =\{F\in\sC: F\cap n^{1/\alpha}K\neq\emptyset\}
    =\{n^{1/\alpha}F: F\cap K\neq\emptyset\}
    =n^{1/\alpha}\sC_K.
  \end{displaymath}
  Furthermore, $\nu((\sC_{K_1}\cap\cdots\cap\sC_{K_m})\times A)$ is
  recovered from the restriction of the L\'evy measure
  $M_{K_1,\dots,K_m}$ of $(\Phi(K_1),\dots,\Phi(K_m))$ onto
  $(0,\infty)^m$ and so inherits the homogeneity property of
  $M_{K_1,\dots,K_m}$ and satisfies
  \begin{displaymath}
    \nu((\sC_{n^{1/\alpha}K_1}\cap\cdots\cap\sC_{n^{1/\alpha}K_m})\times
    A)=n\nu((\sC_{K_1}\cap\cdots\cap\sC_{K_m})\times A).
  \end{displaymath}
  Since finite intersections of $\sC_{K_i}$ generate the
  $\sigma$-algebra on $\sC$, the measure $\nu$ is uniquely determined
  and satisfies $\nu(n^{1/\alpha}\sM\times A)=n \nu(\sM\times A)$
  for all $n\geq2$ on the whole $\sigma$-algebra on $\sC$. By a
  standard argument, \eqref{eq:nu-homogeneous} holds
  for all rational $t$.

  Let $\{x_i,i\geq1\}$ be a Poisson point process on $\R^d$ with
  intensity measure having density $\|x\|^{\alpha-d}$ for
  any $\alpha\in(0,\infty)$. The corresponding intensity measure
  satisfies $\mu(tL)=t^{\alpha}\mu(L)$ for each compact set
  $L$. Therefore, the valuation given by \eqref{eq:points}
  is a random measure which is
  $\alpha$-stable in argument, so that all values
  $\alpha\in(0,\infty)$ are possible.

  Next, $\alpha<0$ is not possible if $\Phi$ is $\sigma$-continuous
  and monotone, since then
  $(T_{n^{-1/\alpha}}\Phi)(K)=\Phi(n^{1/\alpha}K)\to \Phi(\{0\})$
  a.s.\ as $n\to\infty$ if $0\in K$. On the other hand, stability
  gives that $\Phi(n^{1/\alpha}K)$ coincides in distribution with 
  $\Phi_1(K)+\cdots+\Phi_n(K)$ for independent copies of
  $\Phi$. If $\Prob{\Phi(K)>0}>0$, then the sum on the right-hand side
  tends to $\infty$ in probability, which is a contradiction. If
  $0\notin K$, replace $K$ with a sufficiently large ball and use
  monotonicity. 

  \noindent
  (ii) \textsl{Sufficiency} is obvious, since \eqref{eq:nu-scaling-value}
  means that the L\'evy measure of $n^{1/\alpha}\Phi$ is exactly
  the L\'evy measure of $\Phi$ multiplied by $n$.

  \noindent
  \textsl{Necessity.}  The stability property implies that the L\'evy
  measure $M_K$ of the random variable $\Phi(K)$ is
  homogeneous. Furthermore,
  $\nu((\sC_{K_1}\cap\cdots\cap\sC_{K_m})\times A)$ is recovered from
  the restriction of $M_{K_1,\dots,K_m}$ onto 
  $(0,\infty)^m$ and so inherits the homogeneity property of
  $M_{K_1,\dots,K_m}$. Since finite intersections of $\sC_{K_i}$
  generate the $\sigma$-algebra on $\sC$, the measure $\nu$ is
  uniquely determined and is homogeneous on the whole
  $\sigma$-algebra.

  Since $\Phi$ is non-trivial, there exists $K$ for which $\Phi(K)$
  is a non-degenerate non-negative strictly $\alpha$-stable random
  variable, and hence necessarily $\alpha\in(0,1)$. Existence for every
  $\alpha\in(0,1)$ follows by taking $\Phi$ to be the restriction to
  $\Kd$ of a non-negative independently scattered $\alpha$-stable
  random measure with a locally finite control measure. 

  \noindent
  (iii) The valuation $\Phi$ is $\beta$-homogeneous if and only if
  the distribution of the Poisson point process $\{(tF_i,r_i),i\geq1\}$
  coincides with the distribution of
  $\{(F_i,t^{-\beta}r_i),i\geq1\}$ for all $t>0$. This is exactly
  the formulated homogeneity property of $\nu$.
\end{proof}

\section{Stationary random valuations}
\label{sec:stat-rand-valu}

A random valuation $\Phi$ is said to be \emph{stationary} if all its
finite-dimensional distributions are translation invariant, that is,
the joint distribution of $(\Phi(K_1),\dots,\Phi(K_m))$ is the same as
the joint distribution of $(\Phi(K_1+x),\dots,\Phi(K_m+x))$ for all
$x\in\R^d$, all $K_1,\dots,K_m\in\Kd$, and all $m\ge1$. Similarly,
$\Phi$ is \emph{isotropic} if all its finite-dimensional distributions
are invariant under rotations applied to the argument. We assume that
random valuations do not have a deterministic part.

\begin{example}[Stationary pure-jump completely random measure]
  \label{ex:poisson-r-m}
  Let $\{(x_i,r_i),i\geq1\}$ be the Poisson point process on
  $\R^d\times(0,\infty)$ with intensity measure equal to the product
  of the Lebesgue measure on $\R^d$ and a measure $\mu$ on
  $(0,\infty)$ such that $\int \min(1,r)\mu(dr)<\infty$. Then $\eta$
  given by \eqref{eq:points}
  is a stationary pure-jump completely random measure. If
  restricted to compact convex sets, it defines a stationary
  $\sigma$-continuous ID valuation with independent increments and
  independent values, see Theorem~\ref{thr:indep-values}.
\end{example}

If $\Phi$ is an ID valuation, then the uniqueness of the L\'evy
measure implies that $\Phi$ is stationary (isotropic) if and only if
its L\'evy measure $\Lambda$ is invariant under translations
(rotations).  The stationarity property of a $\sigma$-continuous
monotone ID valuation with inde\-pen\-dent increments is equivalent to
translation invariance of the measure $\nu$ on $\sC\times(0,\infty)$
(see Theorem~\ref{thr:2}) with respect to translations of the
first component, that is,
\begin{displaymath}
  \nu((\sC_{K_1}\cap\cdots\cap\sC_{K_m})\times A)
  =\nu((\sC_{K_1+x}\cap\cdots\cap\sC_{K_m+x})\times A)
\end{displaymath}
for all $x\in\R^d$, $K_1,\dots,K_m\in\Kd$, $m\geq1$, and Borel
$A\subset(0,\infty)$. 

The family of translation-invariant locally
finite measures on $\sC$ has been characterised in Theorem~5.4.1 from
\cite{ma75}. Below we present its variant for measures on
$\sC\times(0,\infty)$. Denote by $G(d,k)$ the Grassmannian in $\R^d$
that consists of linear subspaces of dimension $k$. For $E\in G(d,k)$,
denote by $\Pi_{E^\perp} L$ the projection of a compact set $L$ onto
the orthogonal complement $E^\perp$ of $E$.

The following result characterises translation invariant L\'evy
measures and so describes monotone stationary $\sigma$-continuous ID
random valuations which have independent increments and vanishing
deterministic part. 

\begin{theorem}
  \label{thr:ma}
  Let $\nu$ be a measure on $\sC\times(0,\infty)$ which satisfies
  \eqref{eq:nu-levy} and such that
  $\nu(\{\emptyset\}\times(0,\infty))=0$. Then $\nu$ is invariant
  under translations on $\sC$ if and only if, for each 
  compact set $L\subset\R^d$ and Borel $A\subset(0,\infty)$, we have
  \begin{equation}
    \label{eq:sum}
    \nu(\sC_L\times A)
    =\sum_{k=0}^{d} \int_A \int_{G(d,k)}
      \E(V_{d-k}(Y_{E,r}+\Pi_{E^\perp}L))G_{k,r}(dE) \mu(dr),
  \end{equation}
  where, for each $k=0,\dots,d$, $r\mapsto
  G_{k,r}$ is a measurable map from
  $(0,\infty)$ to the space of finite measures on
  $G(d,k)$ with the weak topology, and $\mu$ is a $\sigma$-finite measure on
  $(0,\infty)$. Furthermore, for each $E\in G(d,k)$ and $r>0$,
  $Y_{E,r}$ is a non-empty random compact convex set in $E^\perp$ such
  that
  \begin{equation}
    \label{eq:sum-levy}
    \int_{(0,\infty)}\min(1,r)
    \int_{G(d,k)}\E(V_{d-k}(Y_{E,r}+\Pi_{E^\perp}B)) G_{k,r}(dE)\mu(dr)<\infty
  \end{equation}
  for all $k=0,\dots,d$, where $B$ is the unit Euclidean ball in
  $\R^d$, and, for every compact set $L$ in $\R^d$, the map
  \begin{equation}
    \label{eq:map-Er}
    (E,r)\mapsto \E(V_{d-k}(Y_{E,r}+\Pi_{E^\perp}L))
  \end{equation}
  is measurable. 
\end{theorem}
\begin{proof}
  For $k=0,\dots,d$, let $\sZ_k$ denote the
  family of closed convex cylinders with $k$-dimensional direction
  space, namely, each $F\in\sZ_k$ can be written in the form $F=E+C$,
  where $E\in G(d,k)$ and $C$ is a non-empty compact convex subset of
  $E^\perp$. As in the proof of \cite[Theorem~5.4.1]{ma75}, the
  translation-invariant measure $\nu$ is concentrated on
  $\bigcup_{k=0}^d\sZ_k\times(0,\infty)$. This follows from the
  mentioned theorem for $\nu(\cdot\times(1/n,\infty))$ and then taking
  the union over $n$. 

  Denote by $s(C)$ the Steiner point of $C$, see
  \cite[Eq.~(1.31)]{S14}, and put $Y=C-s(C)$ and $x=s(C)$.  Then every
  $F\in\sZ_k$ has a unique representation
  \begin{equation}
    \label{eq:centered-cylinder}
    F=E+Y+x,
  \end{equation}
  where $E\in G(d,k)$, $Y\subset E^\perp$, $s(Y)=0$, and
  $x\in E^\perp$.  Let
  \begin{displaymath}
    \sY_k=\{(E,Y):E\in G(d,k),\;
    Y\subset E^\perp\text{ is non-empty compact convex},\;
    s(Y)=0\}.
  \end{displaymath}
  Equipped with the Borel $\sigma$-algebra inherited from
  $G(d,k)\times\Kd$, the space $\sY_k$ is a standard Borel space.
  Indeed, the map $C\mapsto s(C)$ is continuous in the Hausdorff
  metric, and the conditions $Y\subset E^\perp$ and $s(Y)=0$ are
  Borel. The correspondence in \eqref{eq:centered-cylinder} and its
  inverse are Borel measurable.

  Restrict $\nu$ to $\sZ_k\times(0,\infty)$ and transport this
  restriction under the map $(F,r)\mapsto(E,Y,x,r)$ determined by
  \eqref{eq:centered-cylinder}. Denote the resulting measure by
  $\widetilde\nu_k$. It is a measure on
  \begin{displaymath}
    \sX_k=
    \{(E,Y,x,r):(E,Y)\in\sY_k,\ x\in E^\perp,\ r>0\}.
  \end{displaymath}
  The transported measure $\widetilde\nu_k$ is invariant under
  $(E,Y,x,r)\mapsto (E,Y,x+\Pi_{E^\perp}z,r)$.
  Let $B_{E^\perp}$ denote the unit Euclidean ball in
  $E^\perp$ and, for Borel $H\subset\sY_k\times(0,\infty)$, define
  \begin{displaymath}
    \Theta_k(H)
    =
    \frac{1}{\kappa_{d-k}}
    \widetilde\nu_k
    \bigl(
    \{(E,Y,x,r):(E,Y,r)\in H,\ x\in B_{E^\perp}\}
    \bigr),
  \end{displaymath}
  where $\kappa_{d-k}$ is the volume of the unit Euclidean ball in $\R^{d-k}$.
  For $k=d$, we let $\kappa_0=1$. By translation invariance and the
  uniqueness of Haar measure on each fibre $E^\perp$, a monotone-class
  argument yields
  \begin{equation}
    \label{eq:fibre-decomposition}
    \int_{\sX_k} h(E,Y,x,r)\,\widetilde\nu_k(d(E,Y,x,r))
    =
    \int_{\sY_k\times(0,\infty)}
    \int_{E^\perp} h(E,Y,x,r)\; dx\;
    \Theta_k(d(E,Y,r))
  \end{equation}
  for every non-negative Borel function $h$.

  Let $L\subset\R^d$ be compact. For a cylinder $F=E+Y+x$, we have
  $F\cap L\neq\emptyset$ if and only if
  $(Y+x)\cap\Pi_{E^\perp}L\neq\emptyset$.  Hence
  \begin{align*}
    \int_{E^\perp}
    \1_{{(E+Y+x)\cap L\neq\emptyset}}\,dx
    &=
      V_{d-k}\bigl(\Pi_{E^\perp}L+(-Y)\bigr).
  \end{align*}
  Applying the measurable involution $Y\mapsto-Y$ to $\Theta_k$, and
  denoting the resulting measure again by $\Theta_k$, we obtain from
  \eqref{eq:fibre-decomposition}
  \begin{equation}
    \label{eq:Theta-representation}
    \nu(\sC_L\times A)
    =
    \sum_{k=0}^d
    \int_{\sY_k\times A}
    V_{d-k}(Y+\Pi_{E^\perp}L)\,
    \Theta_k(d(E,Y,r)).
  \end{equation}
  It remains to disintegrate the measures $\Theta_k$. Let
  $\theta_k(A)=\Theta_k(\sY_k\times A)$ for Borel $A\subset(0,\infty)$.
  Each $\theta_k$ is $\sigma$-finite. Indeed, since $Y$ is non-empty,
  \begin{displaymath}
    V_{d-k}(Y+\Pi_{E^\perp}B)\geq\kappa_{d-k},
  \end{displaymath}
  and therefore, for every $n\geq1$,
  \begin{displaymath}
    \kappa_{d-k}\theta_k([n^{-1},\infty))
    \leq
    \nu(\sC_B\times[n^{-1},\infty))
    <\infty
  \end{displaymath}
  by \eqref{eq:nu-levy}. Set $\mu=\theta_0+\cdots+\theta_d$, which is
  a $\sigma$-finite measure on $(0,\infty)$. Since all spaces
  involved are standard Borel, the disintegration theorem yields
  finite measurable kernels $H_{k,r}$ on $\sY_k$ such that
  \begin{displaymath}
    \Theta_k(d(E,Y,r))
    =H_{k,r}(d(E,Y))\,\mu(dr).
  \end{displaymath}
  Let $G_{k,r}$ be the $E$-marginal of $H_{k,r}$, that is, $G_{k,r}(D)
  =H_{k,r}\bigl(\{(E,Y)\in\sY_k:E\in D\}\bigr)$ for Borel $D\subset G(d,k)$.
  Then $r\mapsto G_{k,r}$ is a measurable kernel and
  $G_{k,r}$ is a finite measure for $\mu$-almost every $r$.
  Define the joint measure
  \begin{displaymath}
    \bar{H}_k(dr,dE,dY)=H_{k,r}(d(E,Y))\mu(dr)
  \end{displaymath}
  and its $(r,E)$-marginal
  \begin{displaymath}
    \bar{G}_k(dr,dE)=G_{k,r}(dE)\mu(dr).
  \end{displaymath}
  Redefining $G_{k,r}=0$ on a $\mu$-null set if necessary, we may
  assume that $G_{k,r}$ is finite for every $r>0$. 
  Since these are measures on standard Borel spaces, disintegration
  gives a jointly measurable probability kernel
  $Q_k(r,E,dY)$ such that
  \begin{displaymath}
    \bar{H}_k(dr,dE,dY)=Q_k(r,E,dY)G_{k,r}(dE)\mu(dr).
  \end{displaymath}
  On the set where $G_{k,r}$ vanishes, the kernel $Q_k$ may be
  chosen arbitrarily, so that the assertion holds for every $r$.
  
  Let $Y_{E,r}$ denote a random compact convex set with distribution
  $Q_k(r,E,\cdot)$. Then \eqref{eq:Theta-representation} becomes
  \begin{displaymath}
    \nu(\sC_L\times A)
    =
    \sum_{k=0}^{d}
    \int_A\int_{G(d,k)}
    \E V_{d-k}(Y_{E,r}+\Pi_{E^\perp}L)
    G_{k,r}(dE)\mu(dr),
  \end{displaymath}
  which is \eqref{eq:sum}. Notice that the construction also proves
  the measurability of \eqref{eq:map-Er},
  since
  $(E,Y)\mapsto V_{d-k}(Y+\Pi_{E^\perp}L)$ is Borel and
  $Q_k$ is a measurable probability kernel.
  Finally, taking $L=B$ in \eqref{eq:sum} and using
  \eqref{eq:nu-levy} gives
  \begin{align*}
    \infty
    &>
      \int_{(0,\infty)}
      \min(1,r)\,\nu(\sC_B\times dr)\\
    &=
      \sum_{k=0}^d
      \int_{(0,\infty)}\min(1,r)
      \int_{G(d,k)}
      \E V_{d-k}(Y_{E,r}+\Pi_{E^\perp}B)
      G_{k,r}(dE)\mu(dr).
  \end{align*}
  Since all summands are non-negative, each of them is finite, which
  yields \eqref{eq:sum-levy}.

  Conversely, a measure satisfying \eqref{eq:sum} is translation
  invariant.
  Indeed, for $x\in\R^d$,
  \begin{displaymath}
    \Pi_{E^\perp}(L+x)
    =
    \Pi_{E^\perp}L+\Pi_{E^\perp}x,
  \end{displaymath}
  and translation invariance of Lebesgue measure gives
  \begin{displaymath}
    V_{d-k}
    (Y_{E,r}+\Pi_{E^\perp}(L+x))
    =
    V_{d-k}(Y_{E,r}+\Pi_{E^\perp}L).
  \end{displaymath}
  Thus
  $\nu(\sC_{L+x}\times A)=\nu(\sC_L\times A)$ for all compact
  $L$, all $x\in\R^d$, and all Borel $A\subset [n^{-1},\infty)$ with any
  $n\geq 1$. Letting $n$ grow and using the fact that 
  compact hitting sets form a measure-determining class on $\sC$, this
  implies translation invariance of $\nu$.
\end{proof}

It is important to stress that \eqref{eq:sum} provides a
representation of the values of $\nu$ on $\sC_L\times A$ for all (not
necessarily convex) compact sets $L$ and Borel $A\subset[n^{-1},\infty)$
with any sufficiently large $n$. In particular, \eqref{eq:sum}
yields the values of $\nu$ on
$(\sC_{K_1}\cup\cdots\cup\sC_{K_m})\times A=\sC_{K_1\cup\cdots\cup
  K_m}\times A$ for all collections $K_1,\dots,K_m\in\Kd$ and then by
the inclusion-exclusion formula for all
$(\sC_{K_1}\cap\cdots\cap\sC_{K_m})\times A$, hence for all measurable
$\sM\subset\sC$.

It is instructive to separate out in \eqref{eq:sum} the terms
corresponding to the trivial Grassmannians (for $k=0$ and $k=d$), and
then
\begin{align*}
  \nu(\sC_L\times A)
  &=\int_A \E V_d(Y_{0,r}+L)g_0(r)\mu(dr)
    +\1_{L\neq\emptyset}\int_A g_d(r)\mu(dr)\\
  &\quad +\sum_{k=1}^{d-1} \int_A \int_{G(d,k)}
  \E(V_{d-k}(Y_{E,r}+\Pi_{E^\perp}L))G_{k,r}(dE) \mu(dr),
\end{align*}
where $g_0$ and $g_d$ are two non-negative measurable functions.  This
representation shows that the Poisson point process $\{(F_i,r_i), i\geq1\}$
with the intensity measure $\nu$ can be obtained as the superposition
of the stationary Poisson point process of compact convex sets with
marks in $(0,\infty)$ for
$k=0$, the Poisson process $\{(\R^d,r_i),i\geq1\}$ for $k=d$, and
the Poisson processes of random cylinders with $k$-dimensional
direction spaces and compact convex bases in $E^\perp$, $k=1,\dots,d-1$.

\begin{example}
  The term with $k=d$ in \eqref{eq:sum} corresponds to
  $\nu(\sC_L\times A)=\1_{L\neq\emptyset}\theta(A)$ for a L\'evy
  measure $\theta$ on $(0,\infty)$. Then $\Phi(K)=\xi$ for all
  non-empty $K\in\Kd$, where $\xi$ is the infinitely divisible random
  variable with the L\'evy measure $\theta$.

  If $\mu=\delta_1$, then the term corresponding to $k=0$ results in
  the L\'evy measure
  \begin{displaymath}
    \nu(\sC_L\times A)=c\E V_d(Y_{0}+L)\delta_1(A)
  \end{displaymath}
  for $c=G_{0,1}(\{0\})$.
  The corresponding ID valuation $\Phi$ is given by 
  \begin{displaymath}
    \Phi(K)=\sum_{i} \1_{(Y_i+x_i)\cap K\neq\emptyset},
  \end{displaymath}
  where $\{(Y_i,x_i),i\geq1\}$ is the Poisson point process on
  $\sC\times\R^d$ with the intensity being $c$ times the product of the
  distribution of a random compact convex set $Y_0$ (called the
  typical grain) 
  and the Lebesgue measure. In other words,
  $\Phi(K)$ is the number of sets from the germ-grain model which
  intersect $K$, see \cite[Section~4.3]{schn:weil}.  The case of all
  $Y_i=M$ with $s(M)=0$ appears in Example~\ref{ex:M+x}.
\end{example}

In the stationary isotropic case we get the following result. Denote
\begin{displaymath}
  c_{i}^{k}=\frac{\Gamma((k+1)/2)}{\Gamma((i+1)/2)}
  \quad \text{and}\; c_{i,j}^{k,m}=c_{i}^{k}c_{j}^{m}.
\end{displaymath}

\begin{theorem}
  \label{thr:isotropic}
  If $\Phi$ is a $\sigma$-continuous non-negative stationary isotropic
  ID valuation with independent increments, then the corresponding
  L\'evy measure on $\sC\times(0,\infty)$ is given by
  \begin{equation}
    \label{eq:sum-isotropic}
    \nu(\sC_L\times A)
    =\sum_{k=0}^{d} \int_A \int_{G(d,k)}
    \E(V_{d-k}(Y_{E,r}+\Pi_{E^\perp}L))\gamma_k(dE) a_{k,r}\mu(dr)
  \end{equation}
  for each compact set $L$ in $\R^d$ and Borel $A\subset(0,\infty)$,
  where $\gamma_k$ is the normalised Haar measure on $G(d,k)$,
  $Y_{E,r}$ is a random compact convex set in $E^\perp$ such that
  $Y_{RE,r}$ has the same distribution as $RY_{E,r}$ for each
  deterministic rotation $R$ and for almost every relevant $(E,r)$, and $a_{k,r}$ is a measurable non-negative
  function of $r>0$. If $K\in\Kd$, then
  \begin{equation}
    \label{eq:isotropic-K}
    \nu(\sC_K\times A)=\sum_{k=0}^{d} \sum_{j=0}^{d-k}
    c_{0,d}^{d-k-j,k+j} 
    V_{d-k-j}(K) \int_A \E V_j(Y_{E,r})a_{k,r}\mu(dr).
  \end{equation}
\end{theorem}
\begin{proof}
  The measure $\nu$ is invariant under rotations and
  translations. Referring to \cite[Section~5.4]{ma75} or examining the
  proof of Theorem~5.4.1 ibid.\ shows that $G_{k,r}=a_{k,r}\gamma_k$,
  where $a_{k,r}\ge0$ and $\gamma_k$ is the Haar measure on $G(d,k)$,
  and that the random compact convex set $Y_{E,r}$ for $E\in G(d,k)$
  is such that $Y_{RE,r}$ has the same distribution as $RY_{E,r}$ for
  each deterministic rotation $R$. For this, we can redefine the
  kernels on a $\mu$-null set. This yields
  \eqref{eq:sum-isotropic}.

  If $L=K$ is convex, averaging the right-hand side of \eqref{eq:sum}
  with the normalised Haar measure on the family of all rotations
  applied to $K$ and referring to Theorem~6.1.1 from \cite{schn:weil}
  yield that
  \begin{displaymath}
    \nu(\sC_K\times A)
    =\sum_{k=0}^{d} \int_A \int_{G(d,k)} \sum_{j=0}^{d-k}
    c_{0,d-k}^{j,d-k-j} \E V_j(Y_{E,r}) V_{d-k-j}(\Pi_{E^\perp}K)
    \gamma_k(dE) a_{k,r} \mu(dr).
  \end{displaymath}
  By \cite[Theorem~6.2.2]{schn:weil},
  \begin{displaymath}
    \int_{G(d,k)} V_{d-k-j}(\Pi_{E^\perp}K) \gamma_k(dE)
    =c_{d,j}^{d-k,k+j} V_{d-k-j}(K). \qedhere
  \end{displaymath}
\end{proof}

\begin{proposition}
  Assume that $\Phi$ is a stationary $\sigma$-continuous monotone ID
  valuation with independent increments and vanishing deterministic
  part which is also almost surely
  simple, that is, $\Phi(K)=0$ a.s.\ for all $K\in\Kd$ of dimension at
  most $d-1$. Then $\Phi$ is a stationary pure-jump completely random
  measure.
\end{proposition}
\begin{proof}
  Let
  $L_j=B\cap e_j^\perp$, $j=1,\dots,d$, where $B$ is the unit
  Euclidean ball and $e_1,\dots,e_d$ are the coordinate vectors. Each
  $L_j$ has dimension $d-1$, and hence the simplicity assumption
  yields $\Phi(L_j)=0$ a.s.  Consequently, the L\'evy measure of
  $\Phi(L_j)$ vanishes, and so $\nu(\sC_{L_j}\times A)=0$ 
  for every Borel $A\subset(0,\infty)$.
  Applying representation~\eqref{eq:sum} to $L_j$ which yields zero
  value and noticing that all summands are non-negative show that
  \begin{equation}
    \label{eq:simple-zero}
    \int_A\int_{G(d,k)}
    \E V_{d-k}(Y_{E,r}+\Pi_{E^\perp}L_j)
    G_{k,r}(dE)\mu(dr)=0
  \end{equation}
  for every $j=1,\dots,d$ and $k=0,\dots,d$ and for all Borel
  $A\subset(0,\infty)$.

  We first consider $1\leq k\leq d-1$. For every $E\in G(d,k)$, there
  exists $j\in\{1,\dots,d\}$ such that $e_j\notin E^\perp$, since
  $E^\perp$ is a proper linear subspace of $\R^d$. For such $j$, the
  orthogonal projection of $e_j^\perp$ onto $E^\perp$ is
  surjective. Hence, the dimension of $\Pi_{E^\perp}L_j$ is $d-k$, and
  $\Pi_{E^\perp}L_j$ has non-empty interior in $E^\perp$.  Since
  $Y_{E,r}$ is non-empty,
  \begin{displaymath}
    V_{d-k}(Y_{E,r}+\Pi_{E^\perp}L_j)>0.
  \end{displaymath}
  It follows that, for every $E\in G(d,k)$,
  \begin{displaymath}
    \sum_{j=1}^d
    \E V_{d-k}(Y_{E,r}+\Pi_{E^\perp}L_j)>0.
  \end{displaymath}
  Summing \eqref{eq:simple-zero} over $j=1,\dots,d$ therefore shows
  that $G_{k,r}=0$ as a measure on $G(d,k)$ for $\mu$-almost every
  $r$. Thus, all components with $k=1,\dots,d-1$ vanish.

  If $k=d$, then $E^\perp=\{0\}$ and, since $Y_{E,r}$ is non-empty,
  $V_0(Y_{E,r}+\Pi_{E^\perp}L_j)=1$.  Hence \eqref{eq:simple-zero}
  implies that $G_{d,r}=0$ for $\mu$-almost every $r$, consequently,
  the 
  component with $k=d$ also vanishes.
  
  It remains to consider $k=0$. In this case $E=\{0\}$ and
  $E^\perp=\R^d$. Since $G(d,0)=\{\{0\}\}$, write
  $g_0(r)=G_{0,r}(\{\{0\}\})$. We claim that $Y_{0,r}$ is a singleton almost surely
  for almost every $r$ with respect to the corresponding intensity
  measure. Suppose otherwise. Then, with positive probability, $Y_{0,r}$
  contains a non-degenerate segment with direction $u\neq0$. There
  exists $j\in\{1,\dots,d\}$ such that $\langle u,e_j\rangle\neq0$. Since
  $L_j$ has non-empty relative interior in the hyperplane $e_j^\perp$,
  the Minkowski sum of this segment and $L_j$ has non-empty interior in
  $\R^d$. Therefore, $V_d(Y_{0,r}+L_j)>0$.  Consequently, the sum of
  $V_d(Y_{0,r}+L_j)$ over $j$ is strictly positive 
  whenever $Y_{0,r}$ is non-singleton. Summing
  \eqref{eq:simple-zero} over $j$ for $k=0$ shows that this event has
  zero probability for almost every relevant $r$. Hence $Y_{0,r}$ is
  a singleton a.s.\ for $g_0(r)\mu(dr)$-almost every $r$.
  In the centred parametrisation used in the proof of
  Theorem~\ref{thr:ma}, we have $s(Y_{0,r})=0$. A singleton whose
  Steiner point is the origin is necessarily $\{0\}$. Thus $Y_{0,r}=\{0\}$
  almost surely for almost every relevant $r$.
  
  Hence only the $k=0$ component remains, and its cylinders are
  precisely the translated singletons $\{x\}$, $x\in\R^d$. Therefore,
  the L\'evy measure $\nu$ is concentrated on 
  $\{(\{x\},r): x\in\R^d,r>0\}$. 
  Translation invariance implies that the image of $\nu$ under
  $(\{x\},r)\mapsto(x,r)$ is the product of the Lebesgue measure $\ell$ on
  $\R^d$ and a measure $\theta$ on $(0,\infty)$ satisfying
  \begin{equation}
    \label{eq:theta-levy}
    \int_{(0,\infty)}\min(1,r)\theta(dr)<\infty.
  \end{equation}
  Consequently, $\Phi$ has the same distribution as the restriction to
  $K\in\Kd$ of a stationary pure-jump completely random measure $\eta$
  given by \eqref{eq:points} with
  a Poisson point process $\{(x_i,r_i),i\geq1\}$  on
  $\R^d\times(0,\infty)$ of intensity
  $\ell\otimes\theta$.
\end{proof}

\section{Stationary stable ID valuations and McMullen's decomposition}
\label{sec:stationary-stable-id}

Below we specialise the representation of stationary ID valuations to
those which are $\alpha$-stable in argument and show that the
stability exponent $\alpha$ necessarily belongs to
$\{1,\dots,d\}$. The crucial fact is that the L\'evy measure $\nu$ is
supported by an affine Grassmannian $A(d,k)$ for $k=d-\alpha$. Recall
that $A(d,k)$ is the family of all $k$-dimensional affine subspaces of
$\R^d$.

\begin{theorem}
  \label{thr:stationary-homogeneous}
  Let $\Phi$ be a non-trivial stationary $\sigma$-continuous non-negative ID
  valuation with independent increments and vanishing deterministic
  part. If $\Phi$ is $\alpha$-stable in argument, then 
  $\alpha\in\{1,\dots,d\}$, and its L\'evy measure has representation
  \begin{equation}
    \label{eq:stable-stationary}
    \nu(\sC_L\times A)=\int_A
    \int_{G(d,d-\alpha)} V_{\alpha}(\Pi_{E^\perp}L)
    G_{d-\alpha,r}(dE)\mu(dr)
  \end{equation}
  for all compact sets $L$ in $\R^d$ and Borel
  $A\subset(0,\infty)$. Furthermore, in the corresponding Poisson
  point process $\{(F_i,r_i), i\geq1\}$ all sets $F_i$ are affine subspaces
  of dimension $d-\alpha$. Conversely, for each
  $\alpha\in\{1,\dots,d\}$, representation
  \eqref{eq:stable-stationary} implies that $\Phi$ is
  $\alpha$-stable in argument.  
\end{theorem}
\begin{proof}
  By Proposition~\ref{prop:stable}, $\Phi$ is $\alpha$-stable in
  argument if and only if
  $\nu(\sC_{t^{1/\alpha}L}\times A)=t\nu(\sC_L\times A)$ for all
  compact $L$, rational $t>0$,
  and Borel $A\subset(0,\infty)$. Letting $A$ be a subset of
  $[\eps,\infty)$ with $\eps>0$, $L$ be
  the unit Euclidean ball $B$, and using the Steiner formula,
  \eqref{eq:sum} yields that
  \begin{equation}
    \label{eq:stability}
    \nu(\sC_{t^{1/\alpha}B}\times A)=
    \sum_{j=0}^d \kappa_j t^{j/\alpha}\sum_{k=0}^{d-j}
    \int_A\int_{G(d,k)} \E V_{d-j-k}(Y_{E,r})G_{k,r}(dE)\mu(dr),
  \end{equation}
  where $\kappa_j$ is the volume of the unit ball in $\R^j$.
  Again, the Steiner formula yields that
  \begin{displaymath}
    t\nu(\sC_B\times A)=
    t\sum_{j=0}^d \kappa_j \sum_{k=0}^{d-j}
    \int_A\int_{G(d,k)} \E V_{d-j-k}(Y_{E,r})G_{k,r}(dE)\mu(dr).
  \end{displaymath}
  Thus, only the term with $j=\alpha$ appears on the right-hand side
  of \eqref{eq:stability} and thus $\alpha\in\{1,\dots,d\}$. This
  holds for all Borel $A\subset(0,\infty)$ by monotone exhaustion.
  Therefore, 
  \begin{equation}
    \label{eq:j-not-alpha}
    \sum_{k=0}^{d-j}
    \int_A\int_{G(d,k)} \E V_{d-j-k}(Y_{E,r})G_{k,r}(dE)\mu(dr)=0,
    \quad j\neq \alpha.
  \end{equation}
  All terms of \eqref{eq:j-not-alpha} with $j>\alpha$ (if $\alpha<d$)
  vanish, and so $G_{k,r}=0$ as a measure on $G(d,k)$ for $\mu$-almost
  every $r$ and $k\leq d-\alpha-1$. Looking at the terms with
  $j<\alpha$ and $k=d-j$ shows that $G_{k,r}$ vanishes for
  $\mu$-almost all $r$ and $k\geq d-\alpha+1$. Leaving in
  \eqref{eq:j-not-alpha} only the terms with $k=d-\alpha$ yields that,
  for $j=\alpha-1$,
  \begin{displaymath}
    \int_A\int_{G(d,d-\alpha)} \E V_{1}(Y_{E,r})G_{d-\alpha,r}(dE)\mu(dr)=0,
  \end{displaymath}
  so that $Y_{E,r}$ is a singleton for
  $G_{d-\alpha,r}(dE)\mu(dr)$-almost every $(E,r)$. Then
  \eqref{eq:stable-stationary} holds and all sets $F_i$ are affine
  subspaces of dimension $d-\alpha$.  In particular, if $\alpha=d$,
  then it is possible to combine $G_{0,r}(\{\{0\}\})$ with $\mu(dr)$
  to obtain new measure $\tilde{\mu}$,
  so that $\nu(\sC_L\times A)=\tilde{\mu}(A)V_d(L)$ for all compact
  $L$.
  
  The converse easily follows from \eqref{eq:stable-stationary} by
  writing it for $t^{1/\alpha}L$. 
\end{proof}

Below we address situations in which $\Phi$ is decomposable into a sum
of valuations which are stable in argument. The following result is a
random analogue of McMullen's decomposition theorem for
translation-invariant valuations, see \cite[Theorem~2.0.2]{MR3820854}
and \cite[Theorem~3.1]{MR3289841}.

Two random valuations $\Phi$ and $\Psi$ are said to be \emph{weakly
  equivalent} if the distribution of $\Phi(K)$ coincides with that of
$\Psi(K)$ for all $K\in\Kd$, that is, the one-dimensional
distributions of $\Phi$ and $\Psi$ are identical. If $\Phi$ and $\Psi$
are square-integrable and have independent increments, then their weak
equivalence implies that their covariances agree on pairs of
admissible sets.

\begin{theorem}
  \label{thr:mcm}
  Let $\Phi$ be a $\sigma$-continuous non-negative stationary
  isotropic ID valuation with independent increments and vanishing
  deterministic part. Then there exist a non-negative infinitely
  divisible random variable $\xi$ and independent non-negative
  stationary isotropic ID valuations $\Psi_m$, $m=1,\dots,d$, with
  independent increments, such that
  \begin{displaymath}
    \Psi(K)
    =
    \xi\1_{K\neq\emptyset}
    +
    \sum_{m=1}^d\Psi_m(K),
    \qquad K\in\Kd,
  \end{displaymath}
  is weakly equivalent to $\Phi$. Moreover, $\Psi_m$ is
  $m$-stable in argument for every $m=1,\dots,d$.
  
  More precisely, there exist non-negative measures
  $b_0,\dots,b_d$ on $(0,\infty)$ satisfying \eqref{eq:theta-levy}
  such that
  \begin{equation}
    \label{eq:isotropic-mcm-Levy}
    \nu(\sC_K\times A)
    =
    \sum_{m=0}^d V_m(K)b_m(A),
    \qquad K\in\Kd,
  \end{equation}
  for every Borel $A\subset(0,\infty)$, where
  $V_0(K)=\1_{K\neq\emptyset}$.
\end{theorem}
\begin{proof}
  Recall the representation \eqref{eq:isotropic-K} for the L\'evy
  measure obtained in Theorem~\ref{thr:isotropic}, where, by isotropy,
  $\E V_\ell(Y_{E,r})$ does not depend on the particular
  $E\in G(d,k)$.  Put $m=d-k-j$.  Then 
  \eqref{eq:isotropic-K} can be written as 
  \begin{displaymath}
    \nu(\sC_K\times A)  =  \sum_{m=0}^d V_m(K)b_m(A),
  \end{displaymath}
  where
  \begin{equation}
    \label{eq:b-j}
    b_m(A) =  c_{0,d}^{m,d-m}
    \sum_{k=0}^{d-m}
    \int_A  \E V_{d-k-m}(Y_{E,r})a_{k,r}\mu(dr), \quad m=0,\dots,d,
  \end{equation}
  are non-negative measures on $(0,\infty)$.
  Taking $K$ to be the unit Euclidean ball, and using \eqref{eq:nu-levy}
  yields \eqref{eq:theta-levy}. 

  We now construct independent random valuations corresponding to
  these measures.  Let $\xi$ be a non-negative infinitely divisible
  random variable with vanishing deterministic part and L\'evy measure
  $b_0$. Then $K\mapsto \xi V_0(K) = \xi\1_{K\neq\emptyset}$ 
  gives the degree-zero component.
  
  For $m=1,\dots,d$, let $\rho_m$ be the motion-invariant measure on
  the affine Grassmannian $A(d,d-m)$, normalised so that the Crofton
  formula takes the form
  \begin{equation}
    \label{eq:rho-j-crofton}
    \rho_m(\{F\in A(d,d-m):F\cap K\neq\emptyset\})
    = V_m(K),
    \qquad K\in\Kd.
  \end{equation}
  Let $N_m$, $m=1,\dots,d$, be independent Poisson point processes on
  $A(d,d-m)\times(0,\infty)$ with intensity measures $\rho_m(dF)b_m(dr)$,
  and assume that they are also independent of $\xi$. Define
  \begin{equation}
    \label{eq:Psi-j-crofton}
    \Psi_m(K)
    =
    \int r\1_{F\cap K\neq\emptyset}\,
    N_m(d(F,r)).
  \end{equation}
  The integral is well defined, since
  \begin{align*}
    \int_{A(d,d-m)\times(0,\infty)}
    \min(1,r)\1_{F\cap K\neq\emptyset}
    \rho_m(dF)b_m(dr)
    = V_m(K) \int_{(0,\infty)}\min(1,r)b_m(dr) <\infty.
  \end{align*}
  By construction, $\Psi_m$ is non-negative, $\sigma$-continuous,
  stationary and isotropic, and it has independent increments.

  We next verify stability in argument. The motion-invariant measure
  $\rho_m$ is homogeneous of degree $m$ under dilations of affine
  subspaces. Indeed, writing an affine $(d-m)$-flat as $E+x$, with
  $E\in G(d,d-m)$ and $x\in E^\perp$, its translation component is
  $m$-dimensional Lebesgue measure, hence,
  $\rho_m(t^{1/m}\sM)=t\rho_m(\sM)$ for all $t>0$.  Consequently, for
  every $t>0$,
  \begin{displaymath}
    (\rho_m\otimes b_m)(t^{1/m}\sM\times A)
    = t(\rho_m\otimes b_m)(\sM\times A).
  \end{displaymath}
  By Proposition~\ref{prop:stable}, $\Psi_m$ is $m$-stable in
  argument.  Finally, set
  \begin{displaymath}
    \Psi(K)
    =
    \xi\1_{K\neq\emptyset}
    +
    \sum_{m=1}^d\Psi_m(K).
  \end{displaymath}
  Since the components are independent, the L\'evy measure of the
  non-negative infinitely divisible random variable $\Psi(K)$ is the
  sum of their one-dimensional L\'evy measures. By
  \eqref{eq:rho-j-crofton}, 
  the L\'evy measure of $\Psi(K)$ is
  \begin{displaymath}
    \sum_{m=0}^d V_m(K)b_m(dr)
    = \nu(\sC_K\times dr).
  \end{displaymath}
  Both $\Phi(K)$ and $\Psi(K)$ have
  vanishing deterministic part,
  hence $\Phi$ and $\Psi$ are weakly equivalent.
\end{proof}

The weak equivalence in Theorem~\ref{thr:mcm} cannot in general be
strengthened to equality of finite-dimensional distributions, as the
following example shows.

\begin{remark}
  Imposing only stationarity is not sufficient for the validity of
  Theorem~\ref{thr:mcm}. Let $\Phi(K)=\eta(K+A)$ where $\eta$ is a
  Poisson point process and $A\in\Kd$ is not centrally
  symmetric. Then, for some $K\in\Kd$, $\Phi(K)$ does not have the
  same distribution as $\Phi(-K)$.
  However, each random valuation which is
  stable in argument has representation given by
  Theorem~\ref{thr:stationary-homogeneous}, and this representation is
  invariant in distribution with respect to the central symmetry
  transform. Therefore, while the valuation $\Phi$ is infinitely
  divisible, stationary, monotone, and has independent increments, it
  cannot be represented as the sum of stable random valuations.
\end{remark}

\begin{example}
  \label{ex:Poisson-ball}
  Let $\eta$ be the stationary Poisson random measure on $\R^d$, that is,
  \begin{displaymath}
    \eta(K)=\sum_i \1_{x_i\in K},
  \end{displaymath}
  where $\{x_i,i\geq1\}$ is a homogeneous Poisson point process on
  $\R^d$. Let
  \begin{displaymath}
    \Phi(K)=\eta(K+B),
  \end{displaymath}
  where $B$ is the unit Euclidean ball. This is a stationary
  isotropic ID valuation with independent increments. The corresponding
  measure $\nu$ is supported by the family $\{B+x:
  x\in\R^d\}$. Below we construct a different valuation which shares
  with $\Phi$ all one-dimensional distributions. 

  For $m=0,\dots,d$, let $\rho_m$ be the measure on $A(d,d-m)$
  normalised to satisfy \eqref{eq:rho-j-crofton}. Let
  $b_m=\kappa_{d-m}\delta_1$, and consider the Poisson process
  $\{E_{m,i},i\ge1\}$ on $A(d,d-m)$ of intensity
  $\kappa_{d-m}\rho_m$. Define
  \begin{displaymath}
    \Psi_m(K)=\sum_i \1_{E_{m,i}\cap K\neq\emptyset},\quad K\in\Kd,
  \end{displaymath}
  which is clearly a $\sigma$-continuous stationary ID valuation with
  independent increments which is $m$-stable in argument for
  $m=1,\dots,d$; and, for $m=0$, $\Psi_0(K)$ is a Poisson random
  variable, which does not depend on $K\neq\emptyset$.  Therefore,
  $\Psi_m(K)$ is a Poisson random variable of mean
  $\kappa_{d-m} V_m(K)$, independent for different $m$, and so
  $\Psi(K)=\Psi_0(K)+\cdots+\Psi_d(K)$ is Poisson with mean
  \begin{displaymath}
    \E\Psi(K)=\sum_{m=0}^d \kappa_{d-m}V_m(K)=V_d(K+B).
  \end{displaymath}
  Hence, random variables $\Psi(K)$ and $\Phi(K)$ share the same
  distribution for all $K\in\Kd$. The finite-dimensional distributions
  of these two random valuations are clearly different, e.g.,
  $\Phi(\{x\})$ and $\Phi(\{y\})$ are independent if the singletons
  $x$ and $y$ are at distance at least $2$, while $\Psi(\{x\})$ and
  $\Psi(\{y\})$ are equal.  Indeed, both values equal the random
  multiplicity of the unique element $\R^d$ of $A(d,d)$, which is precisely
  the valuation $\Psi_0$.
\end{example}

\begin{example}[Poisson perimeter]
  For $d=2$, let $\nu=\mu_1\otimes\delta_1$ be the product of the
  motion-invariant measure $\mu_1$ on the affine Grassmannian $A(2,1)$
  normalised so that
  $\mu_1(\{L\in A(2,1): L\cap B\neq\emptyset\})=2\pi$ and the Dirac
  measure $\delta_1$. Then the corresponding ID valuation $\Phi(K)$
  with vanishing deterministic part equals the number of lines from
  the Poisson line process which hit $K$. The expectation of $\Phi$ is
  the perimeter of $K\in\Ktwo$.
\end{example}

\appendix

\section{Valuations with symmetric-independent increments}
\label{sec:conv-indep-incr}

We now characterise ID valuations satisfying a weaker
independence condition. We say that a random valuation $\Phi$ has
\emph{symmetric-independent increments} if the random variables
$\Phi(K)-\Phi(K\cap L)$ and $\Phi(L)-\Phi(K\cap L)$ are independent
for all $K,L\in\Kd$ such that $K\cup L$ is convex. 

\begin{lemma}
  \label{lemma:incr}
  If $\Phi$ is a random valuation with independent increments, then
  $\Phi$ has symmetric-independent increments. 
\end{lemma}
\begin{proof}
  By assumption, $\Phi(K\cup L)-\Phi(L)$ and $\Phi(L)-\Phi(K\cap L)$
  are independent. It remains to notice that $\Phi(K\cup L)-\Phi(L)$
  is a.s.\ equal to $\Phi(K)-\Phi(K\cap L)$ by the additivity property
  of $\Phi$.
\end{proof}

A function $g:\Rd\to\R_+$ is said to be \emph{quasi-concave} if
$\{x:g(x)\geq t\}$ is convex for all $t\in\R_+$. For any $K\in\Kd$
denote
\begin{displaymath}
  g^{\vee}(K)=\sup\{g(x):x\in K\}, 
\end{displaymath}
and let $g^{\vee}(\emptyset)=0$. 

The function $g^\vee$ is a valuation. Indeed, let $K,L,K\cup L\in\Kd$
and assume that $g^\vee(K)\geq g^\vee(L)$. Then
$g^\vee(K\cup L)=g^\vee(K)$, while quasi-concavity of $g$ and
convexity of $K\cup L$ yield $g^\vee(K\cap L)=g^\vee(L)$. Hence
$g^\vee(K\cup L)+g^\vee(K\cap L)=g^\vee(K)+g^\vee(L)$.

\begin{theorem}
  \label{thr:1}
  A monotone $\sigma$-continuous ID valuation $\Phi$ has
  symmetric-independent increments if and only if the outer measure of
  its L\'evy measure $\Lambda$ is supported by valuations $\psi$ each
  of which is of the form
  \begin{displaymath}
    \psi(K)=g^\vee(K)
  \end{displaymath}
  for an upper semicontinuous quasi-concave function $g:\Rd\to\R_+$.
\end{theorem}
\begin{proof}
  By Lemma~\ref{lemma:inf-div-indep}, $\Lambda$ is supported by $\psi$
  such that
  \begin{equation}
    \label{eq:psi-cap}
    \psi(K)=\psi(K\cap L)\quad \text{or}\quad \psi(L)=\psi(K\cap L)
  \end{equation}
  for all $K,L\in\Kd$ with $K\cup L\in\Kd$. Furthermore, all $\psi$
  from the support of $\Lambda$ are monotone and
  $\sigma$-continuous. Taking into account monotonicity of $\psi$, we
  have that
  \begin{displaymath}
    \psi(K\cup L)=\psi(K)+\psi(L)-\psi(K\cap L)=\max(\psi(K),\psi(L)).
  \end{displaymath}
  Define $g(x)=\psi(\{x\})$, $x\in\R^d$. Consider any hyperplane $H$
  which intersects the relative interior of $K$ and so splits $K$ into
  two convex sets $K_1$ and $K_2$ whose union is $K$. Assume that
  $\psi(K_1)=\psi(K)$. Splitting $K_1$ further, we obtain a sequence
  $K_n$, $n\geq1$, which shrinks to a point $x$ and such that
  $\psi(K_n)=\psi(K)$ for all $n$. By $\sigma$-continuity,
  $g(x)=\psi(\{x\})=\psi(K)$ for some $x\in K$. Since
  $\psi(K)\geq \psi(\{y\})=g(y)$ for all $y\in K$ by monotonicity, we
  obtain that $g(x)=g^\vee(K)$.

  If $x_n\to x$, pass to a subsequence along which $g(x_n)\to
  \limsup_{n\to\infty} g(x_n)$. Let $K_n$ be the closed convex hull of $\{x\}$
  and $\{x_j,j\geq n\}$. Then $\psi(K_n)=g^\vee(K_n)\geq \sup_{j\geq
    n} g(x_j)$. By $\sigma$-continuity,
  \begin{displaymath}
    g(x)=\psi(\{x\})\geq \limsup_{n\to\infty} g(x_n),
  \end{displaymath}
  so $g$ is upper semicontinuous.
  
  Consider the segment $[x,y]$ and a point $z$ in its relative
  interior. Then
  \begin{displaymath}
    g^\vee([x,y])+g(z)=g^\vee([x,z])+g^\vee([z,y]).
  \end{displaymath}
  Without loss of generality assume that
  $g^\vee([x,y])=g^\vee([z,y])$. Then $g(z)=g^\vee([x,z])\geq
  g(x)$. Thus, $g(z)\geq \min(g(x),g(y))$, and so $g$ is
  quasi-concave.

  In the other direction, if $g$ is quasi-concave, then
  \eqref{eq:psi-cap} holds, so that $\Phi$ has symmetric-independent
  increments by Lemma~\ref{lemma:inf-div-indep}. Indeed, assume that
  $g^{\vee}(K\cap L)=t$. Then 
  $g^{\vee}(K)\geq t$ and $g^{\vee}(L)\geq t$. If both inequalities
  are strict, choose $s$ such that
  $t<s<\min\{g^\vee(K),g^\vee(L)\}$. Then $g(x)>s$ for some $x\in K$
  and $g(y)>s$ for some $y\in L$. By quasi-concavity $g(z)>s$ for all
  $z$ from the segment $[x,y]$. However, this segment intersects
  $K\cap L$, and we arrive at a contradiction. 
\end{proof}



\section{Moments of random valuations}
\label{sec:moments-rand-valu}

Let $\Phi$ be an integrable random valuation, that is, $\Phi(K)$ is
integrable for all $K\in\Kd$. Then its expectation is a deterministic
valuation $\phi=\E\Phi$, which is translation invariant if $\Phi$ is
stationary and is monotone if $\Phi$ is monotone. The dominated
convergence theorem implies that $\phi$ is $\sigma$-continuous if
$\Phi$ is monotone and $\sigma$-continuous.
Assume that $\Phi$ is square-integrable and define its
covariance function by
\begin{displaymath}
  C(K,L)=\E\big[\Phi(K)\Phi(L)\big]-\E \Phi(K)\E \Phi(L),
  \quad K,L\in\Kd.
\end{displaymath}
Since 
\begin{displaymath}
  0=\E\big[\Phi(M)\big(\Phi(K\cup L)+\Phi(K\cap L)
  -\Phi(K)-\Phi(L)\big)\big],
\end{displaymath}
for $K,L,M\in\Kd$ with $K\cup L\in\Kd$, we have that
\begin{equation}
  \label{eq:KL}
  C(M,K\cup L)+C(M,K\cap L)-C(M,K)-C(M,L)=0.
\end{equation}
Thus, the covariance function is additive in each of its arguments and
so is a \emph{bivaluation}.  Denote
\begin{displaymath}
  \sigma^2(K)=C(K,K)=\E\Phi(K)^2-(\E\Phi(K))^2,\quad K\in\Kd. 
\end{displaymath}

\begin{proposition}
  \label{prop:c-sigma}
  Let $\Phi:\Kd\to\R$ be a square-integrable stochastic process such
  that $K\mapsto \E\Phi(K)$ is a deterministic valuation.
  Then $\Phi$ is a random valuation with 
  uncorrelated increments if and only if 
  \begin{equation}
    \label{eq:c-sigma}
    C(K,L)=\sigma^2(K\cap L),\quad K,L\in\Kd,\; K\cup L\in\Kd, 
  \end{equation}
  and $\sigma^2(\cdot)$ is a monotone valuation.
\end{proposition}
\begin{proof}
  Without loss of generality assume that $\Phi$ is centred, that is,
  $\E\Phi(K)=0$ for all $K$.

  \textsl{Necessity.} If $K,L\in\Kd$ and $K\cup L\in\Kd$, then
  Lemma~\ref{lemma:incr} (adjusted for uncorrelated increments)
  implies that $\Phi(K)-\Phi(K\cap L)$ and
  $\Phi(L)-\Phi(K\cap L)$ are uncorrelated. Noticing that
  $\Phi(K)-\Phi(K\cap L)$ and $\Phi(K\cap L)$ are also uncorrelated,
  we have
  \begin{align*}
    C(K,L)&=\E\big[\Phi(K)\Phi(L)\big]\\
    &=\E\big[(\Phi(K)-\Phi(K\cap L)+\Phi(K\cap L))
    (\Phi(L)-\Phi(K\cap L)+\Phi(K\cap L))\big]\\
    &=\sigma^2(K\cap L),\quad K,L,K\cup L\in\Kd. 
  \end{align*}
  Then \eqref{eq:KL} with $M=K\cup L$ becomes
  \begin{displaymath}
    \sigma^2(K\cup L)+\sigma^2(K\cap L)-\sigma^2(K)-\sigma^2(L)=0,
  \end{displaymath}
  meaning that $\sigma^2:\Kd\to\R_+$ is a non-negative valuation.
  For $L\subset K$,
  \begin{displaymath}
    \sigma^2(K)-\sigma^2(L)
    =\E[(\Phi(K)-\Phi(L)+\Phi(L))^2]-\E\Phi(L)^2
    =\E(\Phi(K)-\Phi(L))^2\geq 0,
  \end{displaymath}
  so that $\sigma^2$ is monotone. 

  \textsl{Sufficiency.} Let $K,L$ be an admissible pair. Then
  \begin{align*}
    \mathrm{Var}(&\Phi(K\cup L)
    +\Phi(K\cap L)-\Phi(K)-\Phi(L))\\
    &=\sigma^2(K\cup L)+\sigma^2(K\cap L)+\sigma^2(K)+\sigma^2(L)
    +2\sigma^2(K\cap L)-2\sigma^2(K)-2\sigma^2(L)-2\sigma^2(K\cap L)\\
    &=\sigma^2(K\cup L)+\sigma^2(K\cap L)-\sigma^2(K)-\sigma^2(L)=0.
  \end{align*}
  For every admissible pait $K,L$, the valuation identity holds almost
  surely.  Let $N\subset M\subset L\subset K$
  be convex compact sets. Then
  \begin{displaymath}
    \E[(\Phi(K)-\Phi(L))(\Phi(M)-\Phi(N))]
    =\sigma^2(M)-\sigma^2(N)
    -\sigma^2(M)+\sigma^2(N)=0. \qedhere
  \end{displaymath}
\end{proof}

\begin{corollary}
  \label{cor:cov-admissible}
  If two square-integrable valuations with independent increments
  share the same one-dimensional distributions, then their covariance
  functions coincide for all admissible pairs $K,L\in\Kd$.
\end{corollary}

The values of $C(K,L)$ on convex bodies $K$ and $L$ whose union is not
convex are in general not functionals of $K\cap L$. For instance,
assume that $C(K,L)=\sigma^2(K\cap L)$ for all $K,L\in\Kd$. Then, for
disjoint $K$ and $L$, we would have $C(K,L)=0$.

\section*{Acknowledgment}

AI was supported by the Swiss National Science Foundation, Grant
No.~229505. The authors acknowledge discussions with Tommaso
Vison\`a. 





\end{document}